\documentclass[11pt]{article}

\usepackage{amsfonts}
\usepackage{amsthm}
\usepackage{amsmath}
\usepackage{graphicx}
\usepackage[usenames,dvipsnames]{xcolor}
\usepackage{enumerate}
\usepackage{extarrows}
\usepackage{hyperref}
\newtheorem{theorem}{Theorem}

\newtheorem{lemma}{Lemma}
\newtheorem{remark}{Remark}

\newcommand{\pibar}{\overline{\Pi}}
\newcommand{\pibarinv}{\overline{\Pi}^{\leftarrow}}
\newcommand{\pibarpinv}{\overline{\Pi}^{+,\leftarrow}}
\newcommand{\pibarminv}{\overline{\Pi}^{-,\leftarrow}}
\newcommand{\pibarpminv}{\overline{\Pi}^{\pm,\leftarrow}}

\newcommand{\lambar}{\overline{\Lambda}}

\newcommand{\lampinv}{\overline{\Lambda}^{+,\leftarrow}}
\newcommand{\lamminv}{\overline{\Lambda}^{-,\leftarrow}}

\newcommand{\vsig}{\varsigma}
\newcommand{\vphi}{\varphi}
\newcommand{\veps}{\varepsilon}
\newcommand{\topr}{\stackrel{\mathrm{P}}{\longrightarrow}}
\newcommand{\todr}{\stackrel{\mathrm{D}}{\longrightarrow}}

\newcommand{\eqdr}{\stackrel{\mathrm{D}}{=}}
\newcommand{\R}{\Bbb{R}}
\newcommand{\N}{\Bbb{N}}

\newcommand{\rmd}{{\rm d}}
\newcommand{\rmi}{{\rm i}}
\newcommand{\halmos}{\quad\hfill\mbox{$\Box$}}

\newcommand{\LLL}{{\cal L}}
\newcommand{\YYY}{{\cal Y}}
\newcommand{\III}{{\cal I}}
\newcommand{\JJJ}{{\cal J}}

\newcommand{\RR}{\mathbb{R}}

\newcommand{\XX}{\mathbb{X}}

\newcommand{\wt}{\widetilde}
\newcommand{\wh}{\widehat}

\newcommand{\dto}{\downarrow}

\newcommand{\EEEE}{\mathfrak{E}}

\newcommand{\cadlag}{c\`adl\`ag}

\newcommand{\be}{\begin{equation}}
\newcommand{\ee}{\end{equation}}
\newcommand{\bea}{\begin{eqnarray}}
\newcommand{\eea}{\end{eqnarray}}
\newcommand{\bean}{\begin{eqnarray*}}
\newcommand{\eean}{\end{eqnarray*}}
\newcommand{\ben}{\begin{equation*}}
\newcommand{\een}{\end{equation*}}
\newcommand{\ba}{\begin{aligned}}
\newcommand{\ea}{\end{aligned}}

\def\nexto{\kern -0.54em}

\def\topr{\buildrel P \over \to }

\newcommand{\sgn}{\mbox{sign}}

\numberwithin{equation}{section}
\numberwithin{theorem}{section}
\numberwithin{corollary}{section}
\numberwithin{proposition}{section}
\numberwithin{lemma}{section}

\begin{document}

\title{\bf Convergence of Trimmed L\'evy Processes to Trimmed Stable Random Variables at $0$}
\author{Yuguang Fan \footnote{Email : \href{mailto:yuguang.fan@unimelb.edu.au}{yuguang.fan@unimelb.edu.au}} \\
School of Mathematics and Statistics,
The University of Melbourne,
Australia\\
 }

\maketitle

\begin{abstract}

Let $^{(r,s)}X_t$ be the L\'evy process $X_t$ with the $r$ largest jumps and $s$ smallest jumps up till time $t$ deleted and let $^{(r)}\wt X_t$ be $X_t$ with the $r$ largest jumps in modulus up till time $t$ deleted.
We show that $({}^{(r,s)}X_t - a_t)/b_t$  or $({}^{(r)}\wt X_t - a_t)/b_t$ converges to a proper nondegenerate nonnormal limit distribution as $t \dto 0$ if and only if $(X_t-a_t)/b_t $ converges  as $t \dto 0$ to an $\alpha$-stable random variable, with $  0 <\alpha<2 $, where $a_t$ and $b_t>0$ are non stochastic functions in $t$. Together with the asymptotic normality case treated in \cite{fan2014an}, this completes the domain of attraction problem for trimmed L\'evy processes at $0$.
 
\end{abstract}

\section{Introduction and Main Result}
Let $(X_t)_{t\ge 0}$ be a real valued L\'evy process with canonical triplet $(\gamma,\sigma^2,\Pi)$,
thus having characteristic function $Ee^{\rmi \theta X_t}= e^{t\Psi(\theta)}$, $t\ge 0$,      
$\theta \in \R$, with characteristic exponent
\ben\label{ce}
\Psi(\theta):= \rmi \theta \gamma - \frac 1 2 \sigma^2 \theta^2 +\int_{\R_*}
\left(e^{\rmi \theta x}-1-\rmi \theta x{\bf 1}_{\{|x|\le 1\}}\right)
\Pi(\rmd x),
\een
$\gamma\in\R$, $\sigma^2\ge 0$, $\Pi$ is a Borel measure on $\RR_* : = \RR\setminus \{0\}$, with $\int_{\R_*} (x^2 \wedge 1) \Pi(\rmd x) < \infty$.
The positive, negative and two-sided tails of $\Pi$ are
\ben\label{pidef}
\pibar^+(x):= \Pi\{(x,\infty)\},\ \pibar^-(x):= \Pi\{(-\infty,-x)\},\
{\rm and} \ \pibar(x):=\pibar^+(x)+\pibar^-(x), \ x>0.
\een
The restriction of $\Pi$ to $(0,\infty)$ is $\Pi^+$. Let $\Pi^- = \Pi(-\cdot)$ and $\Pi^{|\cdot|}  = \Pi^+ + \Pi^-$.


Denote the jump process of $X$ by $(\Delta X_t)_{t\ge 0}$, where $\Delta X_t = X_t-X_{t-}$, $t>0$, with  $\Delta X_0 \equiv 0$.
Denote the positive jumps by $\Delta X_t^+ = \Delta X_t \vee 0$ and the negative jumps by $\Delta X_t^- = (-\Delta X_t)\vee 0$.
Note that $(\Delta X_t^+)_{t\ge 0}$ and $(\Delta X_t^-)_{t \ge 0}$ are nonnegative independent processes. 
For any integer $r, s > 0$, let $\Delta X_t^{(r)}$ be the $r^{th}$ largest positive jump and $\Delta X_t^{(s),-}$ be the magnitude of the $s^{th}$ largest negative jump up till time $t$ respectively. We sometimes write $\Delta X_t^{(r),+}$ for $\Delta X_t^{(r)}$.
We write $\wt{\Delta X}_t^{(r)}$ to denote the $r^{th}$ largest jump in modulus up to time $t$. For a precise and formal definition of the ordered statistics, allowing tied values, we refer to Buchmann et al.  \cite{bfm14} (Section 2.1).
The trimmed versions of  $X$ are defined as
\begin{equation}\label{trims}
{}^{(r, s)} X_t:= X_t- \sum_{i=1}^r {\Delta X}_t^{(i)} + \sum_{j=1}^s \Delta X_t^{(j),-},
 \quad {\rm and} \quad
{}^{(r)}\wt X_t:= X_t- \sum_{i=1}^r \wt{\Delta X}_t^{(i)},
\end{equation} which are termed asymmetrical trimming and modulus trimming respectively.

For $s = 0$ and $r = 0$, one sided trimmed processes
\begin{equation}\label{1s}
 {}^{(r)} X_t:= X_t- \sum_{i=1}^r {\Delta X}_t^{(i)} ,
 \quad {\rm and} \quad
{}^{(s,-)} X_t:= X_t + \sum_{i=1}^s {\Delta X}_t^{(i),-},
\end{equation} are special cases of asymmetrical trimming.  These comprise all versions commonly referred to as ``light trimming'', i.e. trimming off a bounded number of jumps from the process. 
Set 
\[{}^{(0,0)}X_t = {}^{(0)} \wt X_t ={}^{(0)}  X_t = {}^{(0,-)}  X_t = X_t. 
\]

We are familiar with the idea of trimming in the random walks literature. Trimming seems to be a natural way to assess the effect of extreme values of a certain kind. In the context of a L\'evy process with infinite activity, i.e., when the L\'evy measure is an infinite measure with a singularity at $0$, trimming at small times has the interesting feature that we have an inexhaustible amount of jumps of minute sizes.  This gives a whole new perspective to trimming a L\'evy process at small times, as compared to trimming of random walks. As $t \to \infty$, an increasing number of jumps with bigger magnitudes come into consideration for removal, but as $t \dto 0$, jumps of bigger sizes progressively become ineligible for removal in the trimming procedure. This makes trimming at small times a non trivial effort with no exact large time analogy and promises a fresh perspective in seeking out potential applications. In the small time paradigm, we zoom in to focus on the hierarchy of the very small jumps. 
By such reasoning we could hardly expect a parallel structure between large time and small time results. 

Examples of potential applications are in high frequency finance \cite{aitJacod2009}, scattering of photons \cite{DavisMarshak1997} and particle physics \cite{zhengetal2013}. In \cite{DavisMarshak1997} and \cite{zhengetal2013}, in particular, and in many other applications areas, ``L\'evy flights'' (processes with heavy tailed increment distributions) are found to accurately describe many physical processes. 

In the special case when the trimmed processes, with appropriate centering and norming, converge to a normal or degenerate distribution, it is shown in \cite{fan2014an} that the original L\'evy process, after being centered and normed with the same functions, will converge to the same normal or degenerate law as $t \dto 0$.  This implies that light trimming, i.e. trimming off a finite number of jumps, has no effect on asymptotic normality or degeneracy.  The next natural question to ponder is whether such consistency holds for limiting stable laws. This motivates the investigation performed in this paper.


In Buchmann et al. \cite{bfm14}, representation formulae for the positively trimmed process and the modulus trimmed process with its corresponding ordered jumps are derived at each finite time $t> 0$.  Having neither independent increments nor time homogeneity, the trimmed process is no longer a L\'evy process. But the law of a trimmed process at any finite time $t > 0$ can be represented as a mixture of a truncated process plus a Poisson number of ties (depending on the atoms of the L\'evy measure of the untrimmed process) with a Gamma random variable. This representation is extended to asymmetrical trimming in Fan \cite{fan2014an} where $r$ positive jumps and $s$ negative jumps are removed from the process.

When the original process is in the domain of attraction of a stable law at $0$, the truncated processes appearing in the representations, both asymmetrical and modulus types, also converge to a nondegenerate infinitely divisible limit random variable with the same centering and norming functions. We can even write out explicitly the characteristic triplet of the limit distribution.  However when taking the Poisson number of ties into consideration, a finite limit can only be reached through a further subsequence $t_k \dto 0$ for certain ranges of truncation levels (see Lemma 2.1 in Fan \cite{fan2014an}).


The domain of attraction of a stable law for L\'evy processes at small times has been completely characterised, for example see Maller and Mason \cite{MM08}, \cite{MM09}, \cite{MM2010}, Doney and Maller \cite{dm02}, Doney \cite{doney04}. Various equivalent analytical conditions are derived in the above references. For $X$ to be in the domain of attraction of an $\alpha$ stable law with $0<\alpha<2$ at $0$,  loosely speaking, its L\'evy measure needs to have a regularly varying singularity at $0$ and the limits $\lim_{z \dto 0} \pibar^\pm(z)/\pibar(z)$ must exist.


This study is also a continuation of applying the rich ideas from the precedent discrete random walks literature to the continuous setting in L\'evy processes. But for looking at small time results, a degree of delicacy and meticulous care is needed to turn around the methods from $t \to \infty$ to $t \dto 0$.  Particular attention has to be paid to the treatment of possible tied values in the order statistics of the jumps. This paper hinges on many useful ideas from Kesten \cite{kesten93} where he deals with the same problem in the random walks large time setting.

To eliminate the compound Poisson case, whose small time behaviour is trivial, assume $\pibar(0+) = \infty$ throughout. The statement of the main theorem is as follows. Let $\N_0 : =\{0,1,2,3, \dots\}$ be the set of nonnegative integers.

\begin{theorem}\label{convStb}
Suppose $\pibar^\pm(0+) = \infty$. There exist a nonstochastic function $a_t$ and a nondecreasing function $b_t > 0$ such that, as $t \dto 0$, for any $r,s \in \N_0$,
\begin{equation}\label{convStb2}
S_t : = \frac{X_t-a_t}{b_t} \text{ converges in  distribution as }  t \dto 0   ,
\end{equation}
if and only if

\begin{equation}\label{convStb.asy}
{}^{(r,s)}S_t : = \frac{^{(r,s)} X_t-a_t}{b_t} \text{ converges in  distribution as }  t \dto 0 ,
\end{equation}
or equivalently,
\begin{equation}\label{convStb.mod}
{}^{(r)}\wt S_t : =  \frac{^{(r)}\wt X_t-a_t}{b_t} \text{ converges in  distribution as }  t \dto 0 .
\end{equation}
\end{theorem}
When $r = 0$ or $s =0 $ in \eqref{convStb.asy}, referring to \eqref{1s}, we define the one-sided trimmed process with centering $a_t$ and norming $b_t$ by 
\ben\label{def.one}
 {}^{(s,-)}S_t : = \frac{{}^{(s,-)}X_t -a_t}{b_t} \quad \text{and} \quad {}^{(r)}S_t : = \frac{{}^{(r)}X_t -a_t}{b_t}.
\een

Throughout the paper, \eqref{convStb.asy} and \eqref{convStb.mod} will be written as $Law({}^{(r,s)}S_t) \to G$ and $Law(^{(r)}\wt S_t) \to \wt G$ as $t\dto 0$.

When \eqref{convStb2} holds, the limit distribution could be a degenerate distribution, a normal distribution or a stable law with index in $(0,2)$. In the case of asymptotic degeneracy or normality, by \cite{fan2014an}, the limit distribution of the trimmed process in \eqref{convStb.asy} or \eqref{convStb.mod} is the same. When $X_t$ is in the domain of attraction at $0$ of a stable law with index in $(0,2)$, we can derive that the trimmed process with the same centering and norming will converge to a corresponding ``trimmed stable law'' (see Lemma \ref{easy_dir}). This is the ``easy" direction of Theorem \ref{convStb}. 
In general the limit distribution in \eqref{convStb2} is different from that of \eqref{convStb.asy} or \eqref{convStb.mod} unless the limit is normal or degenerate. 

The converse directions in Theorem \ref{convStb} for non-normal convergence present a much harder problem. Attention was drawn to this problem in the random walk setting by Maller \cite{mal82} and Mori \cite{mori84}, ultimately to be resolved in Kesten \cite{kesten93}. Our main objective in this paper is to address this problem in the L\'evy setting. Because there is no ``small time'' concept for random walks, some quite different methods have to be developed, which may be of use in other applications.
Some particular instances of this are flagged where they occur in the proofs, and see Section \ref{rmcty} where the quadratic variation plays a key role in removing an assumption of continuity.

\bigskip\noindent
\begin{remark} (Domain of attraction of a trimmed stable process)\\
\rm We say that a stochastic process $X_t$ is attracted to a random variable $Y$ at $0$ if there exist nonstochastic functions $a_t \in \R$ and $b_t > 0$ with $b_t \dto 0$ as $t \dto 0$ such that $(X_t - a_t)/b_t \todr Y$. Define an $r,s$-trimmed stable process $({}^{(r,s)}Z_t)_{t\ge 0}$ analogously to \eqref{trims} and \eqref{1s} where $(Z_t)$ is an $\alpha$-stable L\'evy process with index $\alpha \in (0,2)$. We can deduce that there exist nonstochastic functions $a_t$, $b_t$ such that $({}^{(r,s)}X_t -a_t)/b_t \to Y$ if and only if $Y \eqdr {}^{(r,s)}Z_1$. To see this, note that if the trimmed $X$ process converges to a non-normal non-degenerate distribution, then by Theorem \ref{convStb}, the original process also converges, with the same centering and norming functions, to a non-normal non-degenerate random variable $\wh Y$. Then \cite{MM08} (see Theorem 2.3) shows that $\wh Y$ is necessarily a stable law. Then by Lemma \ref{easy_dir} (below), we have the limit random variable $Y \eqdr {}^{(r,s)}Z_1$. This proves the necessity. Conversely, for each ${}^{(r,s)}Z_1$, there exist a L\'evy process $X_t$ and nonstochastic functions $a_t \in \R$ and $b_t > 0$ such that $(X_t -a_t)/b_t \todr Z_1$. Hence by Lemma \ref{easy_dir} again, ${}^{(r,s)}X_t$ is attracted to  ${}^{(r,s)}Z_1$. We call ${}^{(r,s)}Z_1$ an $r,s$-trimmed stable random variable and its distribution an $r,s$-trimmed stable law. So we have shown that for each $r,s$-trimmed stable law, there exists a L\'evy process $X_t$ such that the $r,s$-trimmed L\'evy process ${}^{(r,s)}X_t$ is in its domain of attraction. 
So any $r,s$-trimmed stable distribution has a nonempty domain of attraction. And all possible non-degenerate non-normal limits of normed, centered, $r,s$-trimmed L\'evy processes are $r,s$-trimmed stable distributions. A similar characterisation holds for the modulus trimmed domains of attraction.

\end{remark}

\section{Preliminary Results}

Through this section, we assume \eqref{convStb.asy} or \eqref{convStb.mod}.
We can first eliminate the case when the limit distribution in \eqref{convStb.asy} or \eqref{convStb.mod} is a normal or degenerate law as this case has been thoroughly dealt with in \cite{fan2014an}. It has been proved in \cite{fan2014an} that if \eqref{convStb.asy} or \eqref{convStb.mod} holds with limit distribution being normal or degenerate, $S_t$ also converges to the same law. This is derived in \cite{fan2014an} by first showing that the tightness of ${}^{(r,s)}S_t$ or ${}^{(r)}\wt S_t$ implies the tightness of $S_t$ (see Theorem 1.1 in \cite{fan2014an}). By eliminating the degenerate distribution, we have that \eqref{convStb.asy} or \eqref{convStb.mod} implies that the untrimmed process $X_t$ is in the Feller class (refer to Maller and Mason \cite{MM2010} for more details on properties of Feller class) at $0$, i.e. $S_t$ is stochastically compact as $t \dto 0$. This is also shown in \cite{fan2014an}. By relating to analytical equivalences for the Feller class in terms of the tail of the L\'evy measure and the truncated moments, we can derive bounds for important analytic quantities in the present situation. These quantities are then used to estimate the magnitudes of both the positive and negative tail probabilities of the trimmed process for sufficiently small $t$.

The tail of the marginal distribution of the trimmed process is very hard to compute with the precision needed to prove Theorem \ref{convStb}. Even with the knowledge of the representation formula in \cite{bfm14}, it seems extremely difficult to express the tail probabilities in terms of useful quantities, for example, in terms of the tail of the corresponding L\'evy measure. Hence the idea we pursue is to bound the trimmed process above and below by the distribution of its next largest jump. The distributions of these ordered jumps can be computed directly, for example in Fan \cite{fan2014an}, in terms of the tail of the L\'evy measure and also estimated asymptotically.

Our aim, then, is to show that \eqref{convStb.asy} or \eqref{convStb.mod} implies that $\pibar$ is regularly varying with an index $\alpha \in (0,2)$ at $0$ and also that the limits $\pibar^\pm(z)/\pibar(z)$ exist as $z \to 0$.  It seems to be particularly difficult to prove the latter fact from $\eqref{convStb.mod}$ as the order statistics of the modulus jumps have an  expression entangling both $\pibar$ and $\pibar^\pm$. Once having done this, however, Theorem 2.3 in Maller and Mason \cite{MM08} can be used to show that the untrimmed process $X_t$ is in the domain of attraction of a stable law at $0$.


\subsection{Inequalities for the normed ordered jumps}\label{sect:pre}

Recall that in Fan \cite{fan2014an} Theorem 1.1, it is proved that the tightness of the trimmed process 
${}^{(r,s)}S_t$ for given $a_t$ and $b_t > 0$ implies $\Delta X_t^{(k),\pm}/b_t$ are tight at $0$ for all $ k \in \N$. Note that this implies $b_t \to 0$ as $t \dto 0$.
Therefore, by adding a finite number of tight families, we can easily derive that $S_t$ is tight at $0$. We can write  $\Delta X_t^{(r, \pm)} \eqdr \pibarpminv(\Gamma_r/t)$ for each $r \in \N$, where $\Gamma_r$ is distributed as $Gamma(r,1)$ (see \cite{bfm14} or \cite{fan2014an}) and $\pibarpminv$ denotes the inverse functions.
When $f:(0,\infty)\mapsto [0,\infty)$ is a nonincreasing function, its right-continuous inverse is
\ben
f^\leftarrow(x)=\inf\{y>0: \ f(y) \le x \},\ x>0.
\een
Then for each fixed $v, u > 0$, there exist constants $C_v$ and $C_u$ such that for all sufficiently small $t$, we have
\begin{equation}\label{bbd0b}
\pibarpinv\left(v/t\right)\le b_t C_v, \quad {\rm and} \quad \pibarminv\left(u/t\right)\le b_t C_u.
\end{equation}
To see this, suppose on the contrary that there exist sequences $\{t_k\} \dto 0$ and $\{M_k\} \to \infty$ such that $\pibarpinv(v/t_k)/b_{t_k} > M_k$ for all $k \in \N$. Then for each $k \in \N$, we have
\ben\label{bdd1b}
1-e^{-v} \le  P(\Delta X_{t_k}^{(1)} > \pibarpinv(v/{t_k})) = P\left(\frac{\Delta X_{t_k}^{(1)}}{b_{t_k}} > \frac{\pibarpinv(v/{t_k})}{b_{t_k}} > M_k\right).
\een
Since $\Delta X_t^{(1)}/b_t$ is tight at $0$, the RHS tends to $0$ as $k \to \infty$.
As $v > 0$ is arbitrary, this gives a contradiction which proves the first inequality in \eqref{bbd0b}. The second inequality is proved similarly.
By the same argument, under the assumption that ${}^{(r)}\wt S_t$ is tight, the normed modulus ordered jumps are tight, i.e. 
$\wt{\Delta X}_t^{(k)}/b_t$ is tight for all $k \in \N$ as $t \dto 0$. 
By the same argument, then, for each $v>0$, there exists a $C_v$ such that, for all sufficiently small $t$,
\begin{equation}\label{bbd0}
\pibarinv\left(v/t\right)\le b_t C_v.
\end{equation}
An equivalent analytical condition derived in Fan \cite{fan2014an} for the tightness of all normed ordered jumps $\Delta X_t^{(r),\pm}/b_t$, $r \in \N$, is 
\be\label{bbd1}
 \lim_{x \to \infty} \limsup_{t \dto 0} t\pibar^\pm(xb_t) = 0.
\ee
Then for each $\veps \in (0,1) $, there exists $x_1(\veps) $ large such that 
\be\label{remark1} 
\limsup_{t\dto 0} t\pibar^\pm(x b_t) \le \veps , \quad x > x_1(\veps),
\ee
and there exists $t_1(\veps, x_1)$ small such that for all $x > x_1$ and $0<t < t_1$,
$t\pibar^\pm(xb_t) \le \veps$.
A similar expression is true with $\pibar^\pm$ replaced by $\pibar$. 
Recall from Fan \cite{fan2014an} Lemma 3.1 that the distribution of the $(r+1)^{st}$ largest jump satisfies
\be\label{rJ1}
P(\Delta X_t^{(r+1)}  > y ) = \int_{0}^{t\pibar^+(y)}P(\Gamma_{r+1} \in \rmd v), \quad  y > 0 .
\ee
Hence, we have as lower and upper bounds for the distribution of the ordered jumps
\ben\label{rJ2}
e^{-t\pibar^+(y)} \frac{(t\pibar^+(y))^{r+1}}{(r+1)!} \le P({\Delta X}_t^{(r+1)}  > y ) \le \frac{(t\pibar^+(y))^{r+1}}{(r+1)!} .
\een
Replace $y $ by $xb_t$. By \eqref{remark1}, we can choose $x_1(\veps)$ such that  $t\pibar^+(x b_t) \le \veps \le -\log (1-\veps)$ for $x > x_1(\veps)$ and $t <t_1$. Then,
\begin{equation*}\label{2max_b}
1\ge \frac{P(\Delta X_t^{(r+1)} > x b_t)}{(t\pibar^+(x b_t))^{r+1}/(r+1)!}\ge  e^{-t\pibar^+(xb_t)} \ge 1-\veps, \quad x \ge x_1(\veps), \quad t <t_1.
\end{equation*}
Therefore if \eqref{convStb.asy} or \eqref{convStb.mod} holds in Theorem \ref{convStb}, for any $\veps >0$, we have for each $x >x_1(\veps)$ and all $0<t < t_1$,
\begin{equation}\label{2max0}
\frac{1-\veps}{(r+1)!}\left(t\pibar^+(xb_t)\right)^{r+1}\le P\left(\Delta X_t^{(r+1)} >  x b_t\right)\le \frac{\left(t\pibar^+(xb_t)\right)^{r+1}}{(r+1)!};
\end{equation} and similarly,
\begin{equation*}\label{2max0-}
\frac{1-\veps}{(s+1)!}\left(t\pibar^-(xb_t)\right)^{s+1}\le P\left(\Delta X_t^{(s+1),-} >  x b_t\right)\le \frac{\left(t\pibar^-(xb_t)\right)^{s+1}}{(s+1)!}.
\end{equation*}

\subsection{Eliminate Normal and Degenerate Limits}

We have assumed \eqref{convStb.asy} or \eqref{convStb.mod}, so, as discussed in Section \ref{sect:pre}, $S_t$ is tight.
Recall that \eqref{convStb2} holds if and only if the limit distribution is an $\alpha$-stable $( 0< \alpha < 2)$ or a normal (or degenerate) distribution. It has been proved that Theorem \ref{convStb} holds if the limit random variable is a normal or a degenerate distribution (see Fan \cite{fan2014an}, Theorem 1.2).  
We now want to eliminate the case when $X_t$ is in the domain of partial attraction of a normal law. Suppose this is the case. Then $S_t$ converges to a normal random variable, without loss of generality say $N(0,1)$, through a subsequence. Then the trimmed process ${}^{(r,s)}S_t$ or ${}^{(r)}\wt S_t$ also converges to $N(0,1)$ through the same subsequence, hence by assumption \eqref{convStb.asy} or \eqref{convStb.mod}, we have that ${}^{(r,s)}S_t$ or ${}^{(r)}\wt S_t$ converges to $N(0,1)$ through the whole sequence since we assume that these do have a limit as $t \dto 0$. This reduces to the case that has been studied in \cite{fan2014an}, which we can exclude. Therefore we can assume that $X_t$ is not in the partial domain of attraction of a normal law.

For each $x > 0$, denote the truncated mean and second moment functions by 
\begin{equation}\label{vdef}
 \nu(x) = \gamma - \int_{x <  |y|\le 1} y \Pi(\rmd y), \quad \text{and } \quad V(x) = \sigma^2 + \int_{|y|\le x} y^2 \Pi(\rmd y).
\end{equation}
Now $X_t$ is in the domain of partial attraction of a normal law if and only if
\[ \liminf_{z \dto 0} \frac{z^2 \pibar(z)}{V(z)}  = 0. 
\]
See \cite{fan_thesis} for a proof.
Therefore by eliminating this case we have that
\[ \liminf_{z \dto 0} \frac{z^2 \pibar(z)}{V(z)}  > 0 \quad \text{and} \quad \sigma^2 = 0.
\]

%

In the same way we can also eliminate the case when $S_t$ converges to a degenerate limit through a subsequence. So we can conclude that $X_t$ is in the Feller class at $0$, which is equivalent to (see Theorem 2.1 Maller and Mason \cite{MM2010}),
\[ \limsup_{z\dto 0}\frac{z^2\pibar(z)}{V(z)} < \infty.
\]
From here onwards, in addition to \eqref{convStb.asy} or \eqref{convStb.mod}, we will assume that 
$\sigma^2 = 0$ and there exist constants $0<C_1$, $C_2 < \infty$ such that, for all small $z > 0$, (without loss of generality, say $z \le 1$), we have 
\begin{equation}\label{ass2}
C_1 < \frac{z^2 \pibar(z)}{V(z)} < C_2 , \qquad 0 < z\le 1.
\end{equation}	
Note that $U(z) = V(z) + z^2\pibar(z)$. \eqref{ass2} also implies that 
\begin{equation}\label{ass2b}
0 < \frac{C_1}{1+C_1} < \frac{z^2\pibar(z)}{U(z)} < \frac{C_2}{1+C_2} <\infty \quad \text{for}\quad 0<z\le 1.
\end{equation}

\subsection{Inequalities for the Tail functions and Norming functions}

From \eqref{ass2}, we can derive the following.
\begin{lemma}\label{A3}
Assume \eqref{convStb.asy} or \eqref{convStb.mod} holds, so that $S_{t}$ is tight and \eqref{ass2} holds. Then, for all $0<x<\infty$,
        \begin{equation}\label{ass3}
        0<\liminf_{t \dto 0} t \pibar(x b_t)\le \limsup_{t\dto 0} t\pibar(x b_t)<\infty. 
        \end{equation}
\end{lemma}

\begin{proof}[Proof of Lemma \ref{A3}:] \
Since $S_t$ is tight as $t\dto 0$, by \eqref{bbd1}, we can find an $a_0>0$ such that for all $x\ge a_0$
\begin{equation*}\label{a3pf_a}
\limsup_{t \dto 0} t\pibar(xb_t) < \infty.
\end{equation*} Note that $V(\cdot)$ is a non-decreasing function and $\pibar(\cdot)$ is a non-increasing function. Since $b_t \to 0$, we can choose $t_2=t_2(a_0)$ such that $a_0 b_t\le 1$ for $0<t\le t_2$. Then for any $0<x \le a_0$ and $0<t\le t_2$, we have $xb_t\le 1$,  so by \eqref{ass2},
\begin{equation*}\label{a3pf_b}
C_1 < \frac{x^2b_t^2 t\pibar(xb_t)}{tV(xb_t)} < C_2.
\end{equation*}  Hence for each fixed $x \le a_0$ and $t \le t_2$, 
\begin{align*}\label{a3pf_b2}
t\pibar(xb_t) \le C_2 \frac{tV(xb_t)}{x^2 b_t^2} \le C_2 \frac{tV(a_0 b_t)}{a_0^2 b_t^2}\frac{a_0^2}{x^2} \le \frac{C_2 a_0^2}{C_1 x^2} t\pibar(a_0 b_t) < \infty.
\end{align*}
This proves the right hand inequality in \eqref{ass3} .  
Suppose we have for some $a_0>0$,  
\[ \liminf_{t\dto 0} t\pibar(a_0b_t) = 0.
\] 
Then $\liminf_{t\dto 0} t\pibar(xb_t) = 0$ for all $x > a_0$. For each $x \le a_0$, by \eqref{ass2},
\begin{equation*}\label{a3pf_c}
t\pibar(xb_t) \le \frac{C_2 tV(xb_t)}{x^2 b^2_t}\le \frac{C_2 tV(a_0b_t)}{x^2b^2(t)} \le \left(\frac{C_2 a_0^2}{C_1 x^2}\right) t\pibar(a_0b_t). 
\end{equation*}
This implies that for all $x>0$, $\lim_{ k\to \infty} t_k\pibar(x b(t_k)) = 0$ for some sequence $\{t_k\} \dto 0$.
Since $S_t$ is tight, $S_{t_k}$ converges to a finite random variable $Y$ along a subsequence of $\{t_k\}$, still denote it as $\{t_k\}$. 
Kallenberg's convergence criterion (Theorem 14.15 in \cite{kal02}) states that, $(X_{t_k} - a_{t_k})/b_{t_k} \todr Y$  as $k \to \infty$, where $Y$ is an infinitely divisible random variable with canonical triplet $(\beta, \tau^2, \Lambda)$, if and only if for each continuity point $x > 0$ of $\lambar^\pm$,
\ben
t_k \pibar^\pm (xb_t) \to \lambar^\pm(x)\quad \text{and} \quad
\frac{t_k V(xb_{t_k})}{x^2b^2_{t_k}} \to \tau^2 + \int_{|y|\le x} y^2 \Lambda(\rmd y).
\een
But then the subsequential limit $Y$ has L\'evy measure $0$. Thus $S_{t_k}$ converges to a normal or degenerate distribution, which possibility we have excluded. So
\begin{equation*}
\liminf_{t\dto 0} t\pibar(xb_t) > 0 \quad \text{for all } x>0.
\end{equation*}
This proves the left hand inequality in \eqref{ass3}, completing Lemma \ref{A3}.

\end{proof}

Take $x =1$ in \eqref{ass3}. Then there exist constants $0<C_3, C_4<\infty $, $t_2(1)$ such that 
\begin{align}\label{bdd(b)}
0<C_3 <t\pibar(b_t) <C_4 <\infty \quad \text{for } 0<t\le t_2(1).
\end{align}


The next lemma gives us more bounds from \eqref{ass2}  and Lemma \ref{A3}.
\begin{lemma}\label{bdds}
Assume \eqref{convStb.asy} or \eqref{convStb.mod} holds, so that \eqref{ass2} and \eqref{ass3} hold. Then
\begin{enumerate}[\rm(a)]
\item There exist constants $0 < C_5, D < \infty$ such that for any $\lambda \ge 1$, $0 < z \le 1$ with $\lambda z \le 1$, we have
\begin{equation}\label{lem5}
\frac{\pibar(z)}{\pibar(\lambda z)} \le C_5 \lambda^D. 
\end{equation}
\item For each $ \lambda \ge 1  $ and $ 0 < z\le 1$ such that $\lambda z \le 1$, we have 
\begin{equation}\label{bdd(c)}
\frac{V(\lambda z)}{V(z)}\le (1+C_2)\lambda^\rho, \qquad \text{where } \rho = 2\frac{C_2}{1+C_2}<2.
\end{equation} 
			
\item There exist constants $0<C_6, C_7 <\infty$ with $0<\rho < 2$ defined as in \eqref{bdd(c)} such that for $x\ge 1$ and $t<t_2(x)$, 
\begin{equation}\label{ass9}
C_6 x^{-D} \le t \pibar(xb_t) 
			\le t\pibar(b_t) \frac{C_2(1+C_2)}{C_1} x^{\rho -2} 
			\le C_7 x^{\rho -2}.
\end{equation}

\end{enumerate}
\end{lemma}
\begin{proof}[Proof of Lemma \ref{bdds}:]
(a) Suppose for any $0 < z \le 1$ that $\pibar(z)/\pibar(2z)\le M_1$ for some $1\le M_1 < \infty$. Take $\lambda > 1$ and $k \ge 1$ such that $2^{k-1}<\lambda \le 2^k$. Then 
\begin{align*}\label{5a}
\frac{\pibar(z)}{\pibar(\lambda z)}\le \frac{\pibar(z)}{\pibar(2^k z)} = \frac{\pibar(z)}{\pibar(2z)}\frac{\pibar(2z)}{\pibar(2^2z)}\cdots \frac{\pibar(2^{k-1}z)}{\pibar(2^k z)} \le M_1^k,
\end{align*} for all $2^k z \le 1$.
This gives
\[ M_1^k = 2^{k\log_2M_1} = 2^{\log_2M_1}2^{(k-1)\log_2M_1}\le M_1 \lambda^{\log_2 M_1} = C_5 \lambda^D.
\]
Hence to show \eqref{lem5}, it is sufficient to show that $\pibar(z)/\pibar(2z)$ is bounded for $0<z\le 1/2$. Suppose not. Then there exists a subsequence $\{z_k \dto 0\}$ such that $\pibar(z_k)/\pibar(2z_k) \to \infty$. Since then
\begin{align*}\label{5b}
\frac{z_k^2(\pibar(z_k)-\pibar(2z_k))}{(2z_k)^2 \pibar(2z_k)}= \frac 1 4 \left(\frac{\pibar(z_k)}{\pibar(2z_k)}-1\right)\to \infty,
\end{align*}
we have
\[(2z_k)^2 \pibar(2z_k) = o(z_k^2(\pibar(z_k)-\pibar(2z_k))) 
                        = o(z_k^2 \pibar(z_k))
                         = o(U(z_k))=  o(U(2z_k)). 
 \]
This implies that $(2z_k)^2 \pibar(2z_k) = o(U(2z_k))$, which contradicts \eqref{ass2b}.

\bigskip
(b) We follow a similar argument as Feller \cite{feller67}.  
Let $\lambda \ge 1$ and $\rho = 2C_2/(1+C_2)$. From \eqref{ass2b}, for $0<z < y \le z \lambda \le 1$,  
\[ \frac{2y\pibar(y)}{U(y)}\le \frac{2C_2}{1+C_2} \frac{1}{y} = \frac{\rho}{y}.
\]
Observe that integration by parts gives $U(z)= 2\int_0^z y \pibar(y)\rmd y$. In particular $U$ is absolutely continuous with a.e. derivative $U'(z) = 2 z\pibar(z)$. For $x\ge 1$ and $xz \le 1$,  integrate to get 
\begin{align*}\label{fel_a}
\log\left(\frac{U(z\lambda)}{U(z)}\right)= \int_{z}^{z\lambda}\frac{2y\pibar(y)}{U(y)} \rmd y \le \rho\int_z^{z\lambda}\frac 1 y \rmd y = \rho \log \lambda,
\end{align*} giving $U(z\lambda)/U(z)\le \lambda^\rho$. Then
\begin{equation*}\label{fel_b}
V(z\lambda)\le U(z\lambda)\le \lambda^\rho U(z)=\lambda^\rho[V(z)+z^2\pibar(z)]\le \lambda^\rho V(z)(1+C_2).
\end{equation*} This proves \eqref{bdd(c)}.

\bigskip
(c) For $x\ge 1$, and $t <t_2(x)$ so that $x b_t \le 1$, combine \eqref{bdd(b)} with \eqref{lem5} to get
\begin{equation}\label{ass5} 
t\pibar(xb_t) \ge C_5^{-1}x^{-D}t\pibar(b_t)\ge C_3C_5^{-1}x^{-D}.
\end{equation}
Also by \eqref{ass2} and \eqref{bdd(c)} with $x b_t \le 1$,
\begin{align}\label{ass7}
\pibar(x b_t)&\le C_2 \frac{V(x b_t)}{x^2 b^2_t}
			= C_2 \frac{V(xb_t)}{V(b_t)} x^{-2} \frac{V(b_t)}{b_t^2} 
            \le \frac{C_2}{C_1}(1+C_2)x^{\rho-2}\pibar(b_t).
\end{align} 
By \eqref{ass5}, \eqref{ass7} and \eqref{bdd(b)}, we have
\[ C_3 C_5^{-1} x^{-D} \le t\pibar(x b_t) \le \frac{C_2C_4}{C_1}(1+C_2)x^{\rho-2} : = C_7 x^{\rho-2},
\]hence completing the proof of \eqref{ass9}.

\end{proof}

\begin{lemma}\label{bt}
Suppose \eqref{convStb.asy} or \eqref{convStb.mod} holds. Then the norming function $b_t \equiv b(t)$ satisfies 
\be\label{bt-1}
 b(1/n) \sim b(1/(n+1)).
\ee
\end{lemma}
\begin{remark}\label{trim_norm}
\eqref{bt-1} generalises a similar result in Maller and Mason \cite{MM08} who show that \eqref{bt-1} holds when $S_t$ converges. They show further that the convergence of the untrimmed process to a stable process with $0<\alpha<2 $ implies $b_t \in RV(1/\alpha)$ at $0$. 
\end{remark}
\begin{proof}[ Proof of Lemma \ref{bt}:]
First, we consider only the positively trimmed case, i.e let $^{(r)}S_t = ({}^{(r)}X_t - a_t) / b_t$ converge to a random variable $Y$ in distribution as $t\dto 0$. Let $\lambda > 1$ be fixed. Then $^{(r)}S_{\lambda t} \todr Y$ as $t \dto 0$.   Write
\be\label{bt-2}
{}^{(r)}S_t  = \left(\frac{b_{\lambda t}}{b_t}\right) {}^{(r)}S_{\lambda t} - \frac{X_{\lambda t}- X_t}{b_t} + \sum_{i=1}^r \frac{\Delta X_{\lambda t}^{(i)}-\Delta X_t^{(i)}}{b_t} + \frac{a_{\lambda t} -a_t}{b_t}.
\ee
For each $\veps >0$, 
\begin{align}\label{bt-3}
&P\left(|X_{\lambda t} - X_t -(\lambda-1)t \nu(b_t)| > \veps b_t\right) \nonumber  \\
\le &P\left(|X_{(\lambda-1)t} -(\lambda-1)t \nu(b_t)| > \veps b_t, \, |\wt{\Delta X}_{(\lambda-1)t}^{(1)}| < b_t\right) + P(|\wt{\Delta X}_{(\lambda -1)t}^{(1)}| > b_t) \nonumber \\
\le & \frac{(\lambda-1)tV(b_t)}{\veps^2 b_t^2} +(\lambda-1)t\pibar(b_t).
\end{align} Choose subsequences $t_n = 1/(n+1)$ and $\lambda_n = 1+ 1/n$, then $\lambda_n t_n = 1/n$. Note also that by \eqref{bdd(b)}, we have $t\pibar(b_t) < C_4$. Then 
\[ (\lambda_n -1)t_n \pibar(b_{t_n}) \le \frac{1}{n} C_4 \to 0, \quad \text{as } n \to \infty.
\] Also by \eqref{ass2}, we have for each $\veps > 0$, 
\[ \frac{(\lambda_n -1)t_n V(b_{t_n})}{\veps^2 b^2_{t_n}} \le \frac 1 n \frac{\veps^{-2}}{C_1} t_n \pibar(b_{t_n}) \le \frac 1 n \frac{\veps^{-2}}{C_1} C_4 \to 0, \quad \text{as } n \to \infty.  
\] Therefore, we see that the last line of \eqref{bt-3} tends to $0$ along the subsequences $\{t_n\}$ and $\{\lambda_n\}$. This implies 
\be\label{bt-4}
\frac{X_{\lambda_n t_n} -X_{t_n} -(\lambda_n -1)t_n \nu(b_{t_n}) }{b_{t_n}} \topr 0, \quad \text{as } n \to \infty.
\ee
Next note that for each $i = 1, \ldots ,r$, 
\begin{align}\label{bt-5} 
&P(|\Delta X_{\lambda t}^{(i)} - \Delta X_t^{(i)}| > \veps b_t)
\le P\left( |\Delta X_{\lambda t}^{(i)} - \Delta X_t^{(i)}| > \veps b_t , \, \text{no jump exceeds }   \Delta X_t^{(i)} \text{ on } (t, \lambda t]  \right) \nonumber \\
&\hspace*{2in} + P\left( \text{at least one jump $\Delta X_s$ exceeds } \Delta X_t^{(i)} \text{ for } s \in(t, \lambda t] \right)\nonumber \\
&= 0 + \int_0^\infty \left( 1 - P(\text{no jump exceeds } \pibarpinv(v/t) \text{ on } (t, \lambda t]  \right) P(\Gamma_i \in \rmd v) \nonumber \\
& = \int_0^\infty \left(1 - e^{-(\lambda-1)t\pibar^+(\pibarpinv(v/t))} \right) P(\Gamma_i \in \rmd v)\nonumber \\
&\le (\lambda -1) t \int_0^\infty v P(\Gamma_i \in \rmd v),
\end{align}  
where $\Gamma_i$ is $Gamma(i,1)$  and the last inequality holds because $\pibar^+(\pibarpinv(x))\le x$ and $1-e^{-x} \le x$ for $x > 0$. 
Again choose $t_n = 1/(n+1)$ and $\lambda_n = 1+ 1/n$. Then the RHS of \eqref{bt-5} is less than $(1/n) E(\Gamma_i) \to 0$.
 Therefore we have 
\be\label{bt-6}
 \sum_{i=1}^r \frac{\Delta X_{\lambda_n t_n}^{(i)} - \Delta X_{t_n}^{(i)}}{b_{t_n}} \topr 0 \quad \text{as } n \to \infty.
\ee
Substitute \eqref{bt-4} and \eqref{bt-6} into \eqref{bt-2}, and let 
\[ d_n : = \frac{
a_{1/n} - a_{1/(n+1)} - \nu(b_{1/(n+1)})/n(n+1)}{b_{1/(n+1)}}.
\]  Then we have shown that
\ben\label{bt-7} 
^{(r)}S_{1/(n+1)} = \frac{b(1/n)}{b(1/(n+1))} {}^{(r)}S_{1/n} - d_n+ o_p(1) \to Y \quad \text{as } n \to \infty.
\een 
But $^{(r)}S_{1/n} \to Y$ as well. Applying the convergence of types theorem (see e.g. Gnedenko and Kolmogorov \cite{GK1954} Theorem 10.2), we have both
\ben\label{bt-8}
\frac{b(1/n)}{b(1/(n+1))} \to 1 \quad \text{and}\quad d_n \to 0
 \text { as } n \to \infty. 
\een This completes the proof of \eqref{bt-1} for the positively trimmed process. With a similar argument as in \eqref{bt-5}, we can show that as $n \to \infty$,
\[ \sum_{j=1}^s \frac{\Delta X_{\lambda_n t_n}^{(j),-}- \Delta X_{t_n}^{(j),-}}{b_{t_n}} \topr 0.
\]  Hence \eqref{bt-1} can be proven similarly for the asymmetrically trimmed case. Similarly the same argument holds if we assume \eqref{convStb.mod} instead.
\end{proof}

\section{Proof of Theorem \ref{convStb}: Forward Direction}\label{sect:prf}

First we will deal with the easy direction of Theorem \ref{convStb}. Let $W$ be the limit in distribution of $(X_t-a_t)/b_t$ as $t \dto 0$. If $W$ is a normal or degenerate random variable, by Fan \cite{fan2014an} (see Theorem 1.2), we have that all normed ordered jumps converge to $0$ as $t \dto 0$ and the corresponding trimmed processes converge to the same normal or degenerate distribution. Hence we can assume that $W$ is a non-degenerate and non-normal random variable. By Maller and Mason \cite{MM08} (Theorem 2.3), $W$ is necessarily a stable random variable with index $\alpha\in (0,2)$. Therefore $X_t$ is in the domain of attraction of a stable law. This implies that the tail of the L\'evy measure $\pibar$ is regularly varying with index $-\alpha$ at $0$, i.e. 
\[ \lim_{z \dto 0} \frac{\pibar(xz)}{\pibar(z)} = x^{-\alpha}, \quad x > 0.
\] This further implies that the L\'evy measure has no atoms asymptotically, that is 
\be\label{noay0}
\lim_{z \dto 0}\frac{\Delta \pibar^\pm(z)}{\pibar(z)} = 0
\ee where $\Delta \pibar^\pm(z) = \pibar^\pm(z-)-\pibar^\pm(z)$. To see this, observe that 
\[ 
0 \le \frac{\pibar(z-) - \pibar(z)}{\pibar(z)} \le \frac{\pibar(z(1-\veps))}{\pibar(z)}-1 \stackrel{z \dto 0}{\longrightarrow} (1-\veps)^{-\alpha}-1 \stackrel{\veps \dto 0}{\longrightarrow} 0.
\]

Now we need to introduce the distributional representations from Buchmann et. al. \cite{bfm14} and Fan \cite{fan2014an}. Define three families of processes, indexed by $w>0$, truncating jumps greater than $w$ or smaller than $-w$ from sample paths of $X_t$.
Let $w,t>0$. When $\pibar(0+)=\infty$, we set
\begin{equation}\label{noay0a}
 X_t^{< w}:=X_t-\sum_{0<s\le t} \Delta X_s\;{\bf 1}_{\{\Delta X_s\ge w \}}, \quad
 X_t^{>-w}:=X_t-\sum_{0<s\le t} \Delta X_s {\bf 1}_{\{\Delta X_s\le -w\}},
\end{equation}
and for the modulus case, we truncate jumps with magnitude greater or equal to $w$, i.e.
\begin{equation}\label{noay0b}
\wt X_t^w:=X_t-\sum_{0<s\le t} \Delta X_s\;{\bf 1}_{\{|\Delta X_s|\ge w\}}.
\end{equation}
Recall that the canonical triplet for $X$ is $(\gamma, \sigma^2, \Pi)$. Under the assumption $\pibar(0+)=\infty$,  $(X_t^{<w})_{t\ge 0}$, $(X_t^{ > -w})_{t\ge 0}$ and $(\wt X_t^w)_{t\ge 0}$ are well defined L\'evy processes with canonical triplets, respectively,
\begin{align*}\label{tri_0c}
\left( \gamma  -{\bf 1}_{\{w \le 1\}}\int_{w \le x \le 1}x \Pi(\rmd x),\  \sigma^2, \,  \Pi(\rmd x){\bf 1}_{\{x<w\}} \right),\nonumber \\
\left( \gamma  +{\bf 1}_{\{w \le 1\}}\int_{w \le x \le 1}x \Pi^-(\rmd x),\  \sigma^2, \,  \Pi(\rmd x){\bf 1}_{\{x>-w\}} \right)
\end{align*}
and
\begin{equation}\label{trip2}
\left(\gamma-{\bf 1}_{\{w \le 1\}}\int_{w \le |x| \le 1}x \Pi(\rmd x),\,  \sigma^2,\,  \Pi(\rmd x){\bf 1}_{\{|x|<w\}}\right).
\end{equation}
By Theorem 2.1 in \cite{bfm14} and Section 2 in \cite{fan2014an}, an $r,s$-trimmed process has the following representation. 
Let $(Y^\pm_t)$ be  Poisson processes with unit mean, independent of $(X_t)$ and of each other. Define random variables
\[G^{\pm, w}_t = \pibar^{\pm,\leftarrow}(w) Y^\pm_{t\rho_\pm(w)}  \quad \text{and} \quad 
\rho_\pm(w) = \pibar^\pm(\pibar^{\pm, \leftarrow}(w)-)-w , \quad \text{for each } t, w > 0.
\]
The $G^{\pm,w}_t$ random variables reflect the possibilities of ties among the ordered jumps.
For each $u,v >0$, let $X_t^{u/t,v/t}$ be an infinitely divisible random variable with  
characteristic triplet
\[ \left( \gamma_{u/t,v/t} , 0 , \Pi(\rmd x)_{\{-\pibarpinv(u/t) < x < \pibarpinv(v/t)\}}\right),
\]where
\ben\label{gmauv}
\gamma_{u/t,v/t} =\gamma-{\bf 1}_{\{\pibarpinv(v/t)\le 1\}}\int_{\pibarpinv(v/t)\le x \le 1}x \Pi(\rmd x) + {\bf 1}_{\{\pibarminv(u/t)\le 1\}}\int_{\pibarminv(u/t)\le x \le 1}x \Pi^-(\rmd x).
\een
For each $r, s \in \N$, let $\Gamma_r$ and $\wt \Gamma_s$ be standard Gamma random variables with parameters $r$ and $s$, independent of $(X_t)_{t\ge 0}$,  $(Y^\pm_t)_{t\ge 0}$ as well as each other.  Then for each $t > 0$, we have the following representations for the trimmed processes, asymmetrically,
\begin{align}\label{rand_dis3}
&\left({}^{(r,s)} X_t, \, {\Delta X}_t^{(r)}, \, \Delta X_t^{(s),-}\right) \nonumber \\
&\eqdr \left(  X_t^{u,v}+ G_t^{+,v}- G_t^{-,u},\,  \pibarpinv\left(v\right), \pibarminv\left(u\right)  \right) \bigg|_{v=\Gamma_r/t, u = \wt\Gamma_s/t}.
\end{align}

For each $v>0$, recall the modulus truncated process
$(\wt X^{\pibarinv(v)}_t)_{t\ge 0}$  in  \eqref{noay0b} with canonical triplet defined in \eqref{trip2}.
Then, for each $t>0$ and $r \in \N$, we have the representation
\begin{equation*}\label{2rrep_1}
\left(^{(r)}\wt X_t,\,|\wt{\Delta X}_t^{(r)}|\right)
\eqdr
\left(\wt X_t^{v}+\wt G_t^v, \, \pibarinv\left(v\right)\right) \bigg|_{v=\Gamma_r/t},
\end{equation*} where $\wt G_t^v = \pibarinv(v)(Y^+_{t\kappa^+(v)} - Y^-_{t\kappa^-(v)})$ and
\ben\label{kappm}\kappa^\pm(v) = (\pibar(\pibarinv(v)-)-v)\frac{\Pi\{\pm \pibarinv(v)\}}{\Pi^{|\cdot|}\{\pibarinv(v)\}} {\bf 1}_{\Pi^{|\cdot|}\{\pibarinv(v)\} \neq 0}.
\een

With the above considerations, let's prove the easy direction of Theorem \ref{convStb} in the following lemma.
\begin{lemma}\label{easy_dir}
\eqref{convStb2} implies \eqref{convStb.asy} and \eqref{convStb.mod}.
\end{lemma}

\begin{proof}[Proof of Lemma \ref{easy_dir}:]
From the above analysis, without loss of generality, we can suppose 
$\sigma^2 = 0$ and that the limit random variable $Y$ is infinitely divisible with triplet $(0,0,\Lambda)$ where $\lambar(x) = c x^{-\alpha}$ for some constant $c >0$ and $\alpha \in (0,2)$. From the representation formula in \eqref{rand_dis3}, we have, for each $x > 0$,
\be\label{easy_dir1}
P\left(\frac{{}^{(r,s)}X_t - a_t}{b_t} \le x \right)
= \int_{u,v \in (0,\infty)} P\left(\frac{X_t^{u/t,v/t} + G_t^{+,v/t} - G_t^{-,u/t} - a_t}{b_t} \le x \right)P(\Gamma_r \in \rmd v, \wt \Gamma_s \in \rmd u).
\ee
By separating the events with tied values and without, we get from \eqref{easy_dir1} that 
\begin{align}\label{easy_dir2}
&P\left(\frac{{}^{(r,s)}X_t - a_t}{b_t} \le x \right) \nonumber \\
= &\int_{u,v \in (0,\infty)} P\left(\frac{X_t^{u/t,v/t}- a_t}{b_t} \le x,  G_t^{+,v/t}-G_t^{-,u/t} = 0  \right)P(\Gamma_r \in \rmd v, \wt \Gamma_s \in \rmd u) + \delta_t
\end{align} where
\begin{align*}\label{easy_3}
\delta_t 
&= \int_{u,v \in (0,\infty)} P\left(\frac{X_t^{u/t,v/t} + G_t^{+,v/t} - G_t^{-,u/t} - a_t}{b_t} \le x, G_t^{+,v/t}-G_t^{-,u/t} \neq 0	 \right)P(\Gamma_r \in \rmd v, \wt \Gamma_s \in \rmd u). \nonumber \\
\end{align*}
Now 
\be\label{easy_3a}
 P( G_t^{+,v/t}-G_t^{-,u/t} \neq 0)\le 1-P(Y_{t\rho_+(v/t)} = 0, Y_{t\rho_-(u/t)} = 0 ) 
\ee
in which
\begin{align*}
P(Y_{t\rho_+(v/t)} = 0, Y_{t\rho_-(u/t)} = 0 ) 
= \exp(-t\pibar^+(\pibarpinv(v/t)-)+v)\exp(-t\pibar^+(\pibarpinv(u/t)-)+u).
\end{align*}
By \eqref{noay0}, for each $\veps >0$, we have $\Delta \pibar(z) \le \veps \pibar(z)$ for sufficiently small $z > 0$. Hence, for each $v>0$, for sufficiently small $t > 0$, we have $\pibarpminv(v/t)$ small enough that
\begin{align*}\label{easy_5}
0< t\pibar^\pm(\pibarpminv(v/t)-)-v 
&\le t  \Delta \pibar^+ ( \pibarpinv(v/t) ) + t \Delta \pibar^-(-\pibarminv(v/t))\nonumber \\
& \le \veps t \pibar^+(\pibarpinv(v/t)) + \veps t \pibar^-(\pibarminv(v/t))  \le 2\veps v .
\end{align*} 

Letting $\veps \to 0$ shows that the RHS of \eqref{easy_3a} tends to $0$ as $t \dto 0$ for each $v, u >0$. This shows that $P(  G_t^{+,v/t}-G_t^{-,u/t} = 0 )$ tends to $1$ and $\delta_t \to 0$ as $t \dto 0$. Consequently, we can neglect these terms in \eqref{easy_dir2}.

By assuming \eqref{convStb2}, we have also the convergence of the centered and normed truncated process (see Lemma 2.1 in \cite{fan2014an}), i.e. 
\[\frac{X_t^{u/t,v/t}- a_t}{b_t} \to Y^{u,v}, \quad \text{as } t \dto 0, \quad \text{for each } u, v >0. 
\] where $Y^{u,v}$ is an infinitely divisible random variable with characteristic triplet  $(\beta_{u,v}, 0, \Lambda_{u,v})$ given by
\ben\label{cut_trp13}
\beta_{u,v}= -{\bf 1}_{\{\lampinv(v)\le 1\}}\int_{\lampinv(v)\le y\le 1} y \Lambda(\rmd y) +{\bf 1}_{\{\lamminv(u)\le 1\}}\int_{\lamminv(u)\le y\le 1} y \Lambda^-(\rmd y)  ,
\een and
\[\Lambda_{u,v}(\rmd x)= \Lambda(\rmd x){\bf 1}_{\{ -\lamminv(u)<x<\lampinv(v)\}} \quad \text{for } x \in \R_*.\]
Apply dominated convergence to the integral in \eqref{easy_dir2} to get
\begin{align}\label{easy_6}
&\lim_{t \dto 0}P\left(\frac{{}^{(r,s)}X_t - a_t}{b_t} \le x \right)\nonumber \\
&= \int_{u,v \in (0,\infty)} \lim_{t\dto 0} P\left(\frac{X_t^{u/t,v/t}- a_t}{b_t} \le x \right)P(\Gamma_r \in \rmd v, \wt \Gamma_s \in \rmd u) \nonumber \\
&=\int_{u,v \in (0,\infty)} P\left( Y^{u,v} \le x \right)P(\Gamma_r \in \rmd v, \wt \Gamma_s \in \rmd u) = Y^{u,v}\big|_{u \in \Gamma_r, v \in \wt \Gamma_s} =: W. 
\end{align}
Hence we have proved \eqref{convStb.asy} with $W$ as the limit random variable. 
The proof for \eqref{convStb.mod} is similar. 

\end{proof}

\begin{remark}\label{trimC}
The limit random variable $W$ in \eqref{easy_6} has the distribution of ${}^{(r,s)}Y_1$, where $(Y_t)_{t\ge 0}$ is a stable L\'evy process with canonical triplet $(0,0, \Lambda)$. This can be derived by applying the representation formula (Theorem 2.1 of \cite{bfm14}) again to the stable limit. Hence, 
\[ S_t \to Y \quad \text{implies} \quad {}^{(r,s)}S_t \to {}^{(r,s)}Y_1, \quad \text{as } t \dto 0. 
\] 

An alternative derivation of this is given in \cite{fan_thesis}, where it is shown that the trimming operator as defined in \eqref{trims} is indeed a continuous operator in the space of \cadlag \, functions with respect to Skorokhod's $J_1$ topology.
\end{remark}

\section{Proof of Theorem \ref{convStb}: Converse Direction}

At this stage it is convenient to assume further that 
\be\label{cty.ass}
  \text{the L\'evy measure of $X_t$ is diffuse; that is $\pibar^\pm$ are continuous functions on $(0,\infty)$. }
\ee
Later we will show how to extend the result to full generality.  Assumption \eqref{cty.ass} allows for the following simplification:
\ben\label{cty.ass2}
v = t\pibar^\pm(\pibarpminv(v/t))\le t\pibar^\pm(\pibarpminv(v/t)-) = v, \quad \text{for each } v, t > 0,
\een
which will often be used in what follows. This assumption also means that tied values in the jumps of the $X_t$ occur with $0$ possibility for every $t > 0$. 	\\

%

To proceed, we need both a lower and an upper bound for the tail probabilities of the trimmed process, $P({}^{(r,s)}S_t  > x)$ and $P({}^{(r,s)}S_t  < -x)$,  $x > 0$ in terms of the tails of the corresponding L\'evy measure. We will develop the bounds in Lemmas \ref{lemma6} to \ref{lemma10}. 
Recall that $G$ and $\wt G$ are the limit distributions of $^{(r,s)}S_t$ and ${}^{(r)}\wt S_t$ respectively when $t\dto 0$. 
\begin{lemma}\label{lemma6}
If \eqref{convStb.asy} holds, then for all $\veps >0$, there exist $y_0 = y_0(\veps, G) > 0$ and $x_2(y_0, G) \ge x_1 \ge 1$ such that for all $x >x_2$, $y>y_0$, we have, for sufficiently small $t >0$,
\begin{equation}\label{lem6}
P\big(\Delta X_t^{(r+1)}\ge (x+y)b_t\big) \le (1+\veps)P\big(^{(r,s)}S_t \ge x \big),
\end{equation} and for each $s \in \N$,
\begin{equation}\label{lem6_n}
P\big(\Delta X_t^{(s+1),-}\ge (x+y)b_t\big) \le (1+\veps)P\big(^{(r,s)}S_t \le -x \big).
\end{equation}
If \eqref{convStb.mod} holds, then for all $\veps >0$, $x >x_2$, $y>y_0$ and sufficiently small $t > 0$, 
\begin{equation*}\label{lem6_2a}
P\big(\wt{\Delta X}_t^{(r+1)}\ge (x+y)b_t\big) \le (1+\veps)P\big(^{(r)}\wt S_t\ge x \big),
\end{equation*} and 
\begin{equation}\label{lem6_2b}
P\big(\wt {\Delta X}_t^{(r+1)}\le -(x+y)b_t\big) \le (1+\veps)P\big(^{(r)} \wt S_t \le -x \big).
\end{equation}
\end{lemma}

\begin{proof}[Proof of Lemma \ref{lemma6}:]
Here we only prove \eqref{lem6}, \eqref{lem6_n}-\eqref{lem6_2b} are proved by similar arguments.
Assume \eqref{convStb.asy}. Take $x > 0$, $y > 0$. Then
\begin{align}\label{6a}
P(^{(r,s)}S_t \ge x)&= P\left( ^{(r+1,s)}S_t + \Delta X_t^{(r+1)}/b_t \ge  x \right) \nonumber \\
                  &\ge P\left(^{(r+1,s)}S_t\ge -y,\ \Delta X_t^{(r+1)}\ge (x+y) b_t    \right)\nonumber \\
                  &=P\Big(\Delta X_t^{(r+1)}\ge (x+y)b_t\Big) - P\left( ^{(r+1,s)}S_t < -y,\ \Delta X_t^{(r+1)}\ge (x+y) b_t\right) .   
\end{align}
Recall from the bounds in \eqref{bbd0b} and \eqref{bbd0} that for each $v > 0$ and $t > 0$ sufficiently small, we can choose $x+y \ge C(v, G)$ such that  $(x+y)b_t \ge C(v, G)b_t\ge \pibarpinv(v /t)$. Then
\begin{align}\label{6b}
    &P\left( ^{(r+1,s)}S_t< -y,\ \Delta X_t^{(r+1)}\ge (x+y) b_t\right) \nonumber \\
\le & P\left(\Delta X_t^{(r+1)}\ge (x+y) b_t,\ \Delta X_t^{(r+2)} \ge \pibarpinv(v/t)\right)   \nonumber \\
    &+ P\left( ^{(r+1,s)}S_t< -y,\ \Delta X_t^{(r+1)}\ge (x+y) b_t,\ \Delta X_t^{(r+2)} <\pibarpinv(v/t) \right) \nonumber \\
 = :  & {\rm (I)}+ {\rm (II)}.
\end{align}
We would like to show that both ${\rm (I)}$ and ${\rm (II)}$ are of smaller order than $P\Big(\Delta X_t^{(r+1)}\ge (x+y)b_t\Big)$.
Recall the joint distributional formula in \cite{bfm14} (see their Theorem 2.1), from which we can compute the probability (recall that for any $v,y >0$, $\pibarpminv(v)>y$ iff $\pibar^\pm(y) > v$ and also $\pibar$ is assumed to be continuous)
\begin{align*}\label{6c}
{\rm (I)}
&= P\left(\pibarpinv((\Gamma_{r+1}+\EEEE)/t) > \pibarpinv(v/t),\, \pibarpinv(\Gamma_{r+1}/t) > (x+y)b_t\right) \nonumber \\
&=P\left( \Gamma_{r+1} + \EEEE < t\pibar^+(\pibarpinv(v/t)),\, \Gamma_{r+1} < t\pibar^+((x+y)b_t) \right) \nonumber \\
&=\int_0^{t\pibar^+((x+y)b_t)} (1-e^{-(v-u)})P(\Gamma_{r+1} \in \rmd u) \nonumber \\
& \le\int_0^{t\pibar^+((x+y)b_t)} (v-u) P(\Gamma_{r+1} \in \rmd u) \nonumber \\
&\le v  \int_0^{t\pibar^+((x+y)b_t)} P(\Gamma_{r+1} \in \rmd u)= v P(\Delta X_t^{(r+1)} > (x+y)b_t).
\end{align*}
In the last line we used the representation of the order statistics in \eqref{rJ1}.


Let $\XX$ be the Poisson point process of jumps of $X$ up till time $t$. Hence $\XX$ is defined on $[0,t] \times \R$ with intensity measure $\rmd t \times \Pi(\rmd x)$. Since by assumption \eqref{cty.ass} $\Pi$ is a diffuse measure, then $P(\XX[[0,t] \times \{x\}] > 0) = 0$ for each $x \in \R$. Also by the continuity assumption, we have 
\[ P\left(\XX\big[[0,t] \times (\pibarpinv(v/t),\infty)\big] = r\right) = P \left(\XX\big[[0,t] \times [\pibarpinv(v/t),\infty)\big] = r \right) = \frac{v^r}{r!}e^{-v}.
\]

Now we can write the second term in \eqref{6b} as
\begin{align}\label{6e}
{\rm (II)}   &=P\Big( \ ^{(r+1,s)}X_t-a_t< -yb_t,\ \text{exactly $r+1$ jumps $\Delta X_s$ with $s \le t$ exceed }  (x+y) b_t,\  \nonumber \\
           & \hspace*{3in}\text{and no jump occurs in } \big(\pibarpinv(v/t), (x+y)b_t\big) \Big)\nonumber \\
           &\le P\left( ^{(r+1,s)}S_t< -y \,\big| \, \XX \big[[0,t]\times (\pibarpinv(v/t), \infty) \big] = r+1 \right) \nonumber \\
           & \hspace*{3in} \times  P(\XX \big[[0,t]\times ((x+y)b_t, \infty) \big]= r+1)\nonumber \\
           &\le P\left( \Delta X_t^{(r+1)}\ge (x+y) b_t \right) P( ^{(r+1, s)}S_t < -y \,\big| \,\XX \big[[0,t]\times (\pibarpinv(v/t), \infty) \big] = r+1 ).
\end{align}
Recall the definition of $^{(s,-)}X_t$ in \eqref{1s} and the truncated processes in \eqref{noay0a}. Note that 
\begin{align*}
&P\left( ^{(r,s)}X_t < -y \, , \,\XX \big[[0,t]\times (\pibarpinv(v/t), \infty)\big] = r \right) \nonumber \\
&=P\left( ^{(s,-)}X_t < -y \, , \,\XX \big[[0,t]\times (\pibarpinv(v/t), \infty) \big] = 0 \right) \nonumber \\
& = P\left( {}^{(s,-)}X_t^{<\pibarpinv(v/t)} < -y \right).
\end{align*} 

Now for sufficiently small $t > 0$,
\begin{align*}
&P(\ ^{(r,s)}S_t \le -y) \nonumber \\
&\ge P\left(^{(r,s)}S_t \le -y,\,\,\XX \big[[0,t]\times (\pibarpinv(v/t), \infty) \big]= r   \right)\nonumber \\
&= P\left(^{(s,-)}S_t \le -y\, , \, \XX \big[[0,t]\times (\pibarpinv(v/t), \infty) \big]= 0   \right) \nonumber \\
&= P\left(^{(s,-)}S_t \le -y\, \big| \, \XX \big[[0,t]\times (\pibarpinv(v/t), \infty) \big]= 0   \right)e^{-v}.
\end{align*} Hence 
\begin{align*}
& \limsup_{t \dto 0}P\left(^{(s,-)}S_t \le -y \,\big| \, \XX \big[[0,t]\times (\pibarpinv(v/t), \infty)\big]  = 0  \right) \nonumber\\
 &\le e^{v}\limsup_{t\dto 0} P(^{(r,s)}S_t \le -y) \le e^v \, G(-y+1).
\end{align*} 
Choose $v \le \veps$ and $ y_0 = y_0(\veps, G)$ such that for all $y > y_0$, 
\[
e^{v}G(-y+1)\le \veps. 
\]
Hence the last line of \eqref{6e} is less than, for all sufficiently small $t > 0$, 
\begin{equation*}\label{6g}
\veps P\left( \Delta X_t^{(r+1)}\ge (x+y) b_t \right).
\end{equation*} Substitute the estimates for {\rm (I)} and {\rm (II)} back to \eqref{6a}, to get, for sufficiently small $t > 0$,
\begin{equation*}\label{6h_final}
(1-2\veps)^{-1}P(^{(r,s)}S_t \ge x) \ge P\Big(\Delta X_t^{(r+1)}\ge (x+y)b_t\Big).
\end{equation*} Hence we have shown \eqref{lem6}. 

\end{proof}

%
%
\bigskip
The upper bound of $P(^{(r,s)}S_t > x)$ is more complex. 
First let us introduce two more parameters $\eta $ and $\delta$ such that $ 0<\eta <1 $,  and define $\delta$ in terms of $\eta, t$ and  $x$ to be
\begin{equation}\label{del_def}
\delta  = \delta (\eta, t, x) = \frac{\eta}{24(r\vee s)}\left|\log(t\pibar(x b_t))\right|^{-1}.
\end{equation} Take $\log$s on both sides of \eqref{ass9}. Then for $x\ge x_{3}(\rho, D) \ge x_1 \ge 1$  and $0< t \le t_2(x)$, there exist constants $C_8$ and $C_9$ such that
\begin{equation*}\label{del_eq1}
C_8 \log x \le |\log(t\pibar(xb_t))| \le C_9 \log x.
\end{equation*}
Hence for some constants $0<C_{10}, C_{11}<\infty$ and for $x\ge  x_{3}(\rho, D) \ge 1$ and $t\le t_2(x)$,
\begin{equation}\label{ass10}
C_{10} \frac{\eta}{\log x} \le \delta (\eta, t, x) \le C_{11}\frac{\eta}{\log x}. 
\end{equation} 

%

\bigskip \noindent 
Recall the truncated first moment function $\nu(\cdot)$ in \eqref{vdef}.
Since $S_t$ is tight as $t \dto 0$, by Theorem 14.15 in \cite{kal02}, there exist constants $0<M_2<\infty$ and $t_3(M_2)\le t_2$ such that for all $t\le t_3$,
\begin{equation*}\label{cent1}
a_t - t\nu(b_t) \le M_2 b_t .
\end{equation*} Then for $x > x_4(\eta)\ge x_3$ such that $ x C_{10}\eta /\log x > 1$, we have $\delta x > 1$ by \eqref{ass10} and
\begin{align*}
t\big|\nu(b_t) - \nu(\delta x b_t) \big|  &= t\left|  \int_{b_t< |y|\le 1} y \Pi(\rmd y) -\int_{\delta x b_t< |y|\le 1} y \Pi(\rmd y)    \right| \nonumber \\
                                                                &= t\left|\int_{b_t<|y|\le \delta x b_t} y \Pi(\rmd y)  \right| \nonumber \\
                                                                & \le t \delta x b_t \pibar(b_t) < C_4\delta x b_t, \qquad \text{(from \eqref{bdd(b)})}.
\end{align*}
Then choose $x >  x_5(\eta, M_2) \ge x_4$ such that  $x > 4M_2/\eta$ and $\log x > 4 C_4 C_{11}$, and $t < t_3(M_2)$, so
\begin{align}\label{ctn_2}
\left| a_t-t\nu(\delta x b_t) \right| &\le \left| a_t-t\nu(b_t) \right|+t \left| \nu(b_t) - \nu(\delta x b_t)\right|   \nonumber \\
               &\le (M_2 + C_4 \delta x)b_t 
               \le \left(\frac 1 4 + \frac{C_4 C_{11}}{\log x}\right) \eta x b_t 
               \le \frac 1 2 \eta x b_t. 
\end{align}

\begin{lemma}\label{lemma7}
Suppose \eqref{convStb.asy} holds. Let $0<\eta<1$ and $\veps > 0$. Then for $\delta$ defined as in \eqref{del_def},  for each $x > x_6(\veps, \eta)\ge x_5$ there exists $t_4(\veps,\eta, x)\le t_3$ such that for all $t<t_4$ 
\begin{equation}\label{lem7}
P\left( ^{(r,s)}S_t > \eta x,\ \Delta X_t^{(r+1)} \le \delta x b_t \right)\le \veps \left(t\pibar(xb_t)\right)^{(r\vee s )+1},
\end{equation}
and
\begin{equation}\label{tri2_0}
P\left( ^{(r+1,s)}S_t > \eta x,\ \Delta X_t^{(r+2)} \le \delta x b_t  \right)
\le \veps (t\pibar(xb_t))^{((r+1)\vee s ) + 1} \le \veps (t\pibar(xb_t))^{r+2}.
\end{equation}
Suppose \eqref{convStb.mod} holds. Then under the same conditions, 
\begin{equation}\label{lem7_0b}
P\left( |^{(r)}\wt{S}_t| > \eta x,\ |\wt{ \Delta X}_t^{(r+1)}| \le \delta x b_t \right)\le \veps \left(t\pibar(xb_t)\right)^{(r\vee s)+1}.
\end{equation}
\end{lemma}

\begin{proof}[Proof of Lemma \ref{lemma7}:] 
Here we only prove \eqref{lem7}. \eqref{tri2_0} and \eqref{lem7_0b} can be proved similarly. For $t < t_3$ and $x> x_5(\eta, M_2)$ as in \eqref{ctn_2},
\begin{align}\label{7a}
     &P\left( ^{(r,s)}S_t > \eta x,\  \Delta X_t^{(r+1)} \le \delta x b_t \right) \nonumber \\
 = & P\left( ^{(r,s)}X_t-a_t > \eta xb_t,\  \Delta X_t^{(r+1)} \le \delta x b_t \right)  \nonumber \\
 \le &P\left( ^{(r,s)}X_t - t\nu(\delta x b_t) > \frac 1 2 \eta x b_t,\  \Delta X_t^{(r+1)} \le \delta x b_t\right).
\end{align} The third line comes from the estimate in \eqref{ctn_2}. 
Note that on $\{ \Delta X_t^{(r+1)} \le \delta x b_t\}$,  the truncated process with jumps having magnitude smaller than $\delta x b_t$ is bounded below by the trimmed process as follows. Recall the truncated processes defined in \eqref{noay0a} and \eqref{noay0b}, and write
\begin{align*}
{}^{(r,s)}X_t = {}^{(s-)}X_t - \sum_{i=1}^r \Delta X_t^{(i)} 
	 &\le X_t^{ < \delta x b_t} + \sum_{i=1}^s \Delta X_t^{(i)-} \nonumber \\		
	&\le \wt X_t^{\delta x b_t}+ \sum_{i=1}^s \Delta X_t^{(i)-}{\bf 1}\{ \Delta X_t^{(i)-} \le \delta x b_t\} 
	 \le \wt X_t^{\delta x b_t} + s\delta x b_t.
\end{align*} 
We can choose $x$ big enough that $(r\vee s)C_{11} < \log (x)$, and $(r \vee s) \delta < \frac 1 4 \eta$, and then the last line of \eqref{7a} is less than 
\begin{align}\label{7c}
P\left( \wt X_t^{\delta x b_t} -t\nu(\delta x b_t) \ge \frac 1 4 \eta x b_t    \right).
\end{align}

By compound Poisson approximation, we can write a L\'evy process $X_t$ as the sum of the compensated small jump process and the large jump process as follows, see e.g. Sato \cite{sato99}: for $h >0$ 
\begin{align}\label{7e}
\wt X_t^{h} -t\nu(h) 
&= \lim_{\epsilon \dto 0} \left( \sum_{0<s\le t}\Delta X_s {\bf 1}\{\epsilon < |\Delta X_s|\le h \}-t\int_{\epsilon<|y|\le h} y \Pi(\rmd y) \right)  \nonumber \\
&= : \lim_{\epsilon \dto 0} X_t^{(h)}(\epsilon).
\end{align} Hence setting $h = \delta x b_t$,  for each $\epsilon >0$ and any $\lambda > 0$, by Markov inequality, \eqref{7c} is less than 
\begin{align}\label{7f}
 P\left(X_t^{(h)}(\epsilon) \ge \frac 1 4 \eta x b_t\right)
=P\left(e^{\lambda X_t^{(h)}(\epsilon)} \ge e^{\frac \lambda 4 \eta x b_t}\right) 
\le E\left(e^{\lambda X_t^{(h)}(\epsilon)}\right)  e^{-\frac \lambda 4 \eta x b_t}.
\end{align} By L\'evy-Khinchine formula and also that $e^x-1-x \le e^x x^2/2$, $x>0$, we have
\begin{align}\label{7g}
E\left( e^{\lambda X_t^{(h)}(\epsilon)}\right) &= \exp\left[ t\int_{\epsilon\le |y|\le h}\left(e^{\lambda y} -1-\lambda y {\bf 1}\{ |y|\le 1 \}\right) \Pi(\rmd y) \right] \nonumber \\
&\le  \exp\left( t \int_{\epsilon \le |y|\le h} \frac{(\lambda y)^2}{2}e^{\lambda y} \Pi(\rmd y)\right)\nonumber \\
&\le  \exp\left( t  e^{\lambda h} \lambda^2 \int_{\veps \le |y|\le h} y^2  \Pi(\rmd y)\right) \nonumber \\
&\le  \exp\left( t  e^{\lambda h} \lambda^2 V(h)\right).
\end{align} 
Note for $x>x_4(\eta)$, $\delta x \ge 1$, there exists $t_4(\veps, \eta, x) \le t_3$ such that for $t \le t_4$, $h = \delta x b_t\le 1$. Note that the last line of \eqref{7g} is independent of $\veps$ and by \eqref{7e}, the $\lim_{\veps\dto 0} X_t^{(h)}(\veps)$ exists. 
Choose $\lambda = h^{-1}=(\delta x b_t)^{-1}$. Then by \eqref{7e}, \eqref{7f} and \eqref{7g}, the last line in \eqref{7a} is less than, using \eqref{ass2} and \eqref{ass9},
\begin{align}\label{7h}
\exp\left(  -\frac{\eta}{4\delta} + te\frac{V(\delta x b_t)}{\delta^2x^2b^2_t} \right) 
       \le \exp\left(  -\frac{\eta}{4\delta} + \frac{e}{C_1}t\pibar(\delta x b_t) \right)
        \le \exp\left(  -\frac{\eta}{4\delta} + \frac{C_7 e}{C_1}(\delta x)^{\rho-2}\right).
\end{align} By \eqref{ass10}, recalling that $\rho < 2$, and for $x>x_{4}(\eta)$ we have $\delta x > 1$, hence the upper bound 
\be\label{7ha}
 \frac{C_7 e}{C_1}(\delta x)^{\rho -2} \le \frac{C_7 e}{C_1} 
\le \frac{\log x}{8 C_{11}} \le \frac{\eta}{8\delta}.
\ee

Recalling the definition of $\delta(\eta, t, x)$ in \eqref{del_def}, and also in \eqref{ass9} that $t\pibar(xb_t) \le 1$ for $x\ge 1$ and $t <t_2(x)$, the RHS of \eqref{7h} is bounded by 
\begin{align}\label{7i}
\exp\left( -\frac{\eta}{8\delta} \right) 
&= \exp\left\{ -3(r \vee s) |\log(t\pibar(xb_t))| \right\} \nonumber \\
&= \left(t\pibar(xb_t)\right)^{3(r \vee s)}  \nonumber \\
&=(t\pibar(xb_t))^{(r \vee s)+1} \left(t\pibar(xb_t)\right)^{2 (r \vee s)-1}. 
\end{align} 		
We keep the first term and apply \eqref{ass9} to the second term. Then the last line of \eqref{7i} is less than
\begin{align*}
 & (t\pibar(xb_t))^{(r \vee s)+1}  \left( C_7x^{\rho-2}\right)^{2(r \vee s)-1}\le \veps(t\pibar(xb_t))^{(r \vee s)+1} 
\end{align*}where for $x > x_6(\veps, \eta)\ge x_5$, we have $C_7^{2(r \vee s)-1}x^{(\rho-2)(2(r\vee s)-1)}\le \veps $. 
This completes the proof of the Lemma. 

\end{proof}


\begin{lemma}\label{lemma8}
Let $\veps > 0$, $0<\eta < 1$. If \eqref{convStb.asy} holds, there exists $x_{7}(\veps, \eta) \ge x_6$ such that for each $x>x_{7}$ and all $t<t_4$,
\begin{align}\label{lem8}
P\left(^{(r,s)}S_t\ge x\right)\le P\left( \Delta X_t^{(r+1)} \ge x(1-\eta)b_t  \right) + \veps \left( t \pibar(x b_t)\right)^{r+1}.
\end{align}
If \eqref{convStb.mod} holds, under the same conditions, 
\begin{align*}
P\left({}^{(r)}\wt S_t \ge x \right)\le P\left( \wt{\Delta X}_t^{(r+1)} \ge x(1-\eta)b_t  \right) + \veps \left( t \pibar(x b_t)\right)^{r+1},
\end{align*} and 
\begin{align}\label{lem8_0b}
P\left(^{(r)}\wt S_t\le -x\right)\le P\left( \wt{\Delta X}_t^{(r+1)} \le -x(1-\eta)b_t  \right) + \veps \left( t \pibar(x b_t)\right)^{r+1}.
\end{align}
\end{lemma}

\begin{proof}[Proof of Lemma \ref{lemma8}:]
As before, here we only prove \eqref{lem8} under the assumption \eqref{convStb.asy}.  Choose $0<\eta <1$, then there exists a constant $x_{7}(\eta)\ge x_6(\veps, \eta)$ such that, for all $x>x_{7}$ and $t < t_4$, $ \delta(\eta, t, x) < 1-\eta$ (recall \eqref{ass10}),
we can decompose the event in the LHS of \eqref{lem8} into the following:
\begin{align}\label{8a}
P\left(^{(r,s)}S_t\ge x\right)\le &P\left( \Delta X_t^{(r+1)} \ge x(1-\eta)b_t  \right) \nonumber \\
                                &+ P\left(^{(r,s)}S_t\ge x,\ \Delta X_t^{(r+1)} \le \delta x b_t  \right) \nonumber \\
                                &+ P\left(^{(r,s)}S_t\ge x,\  \delta x b_t<\Delta X_t^{(r+1)} < x(1-\eta)b_t   \right).
\end{align}
By Lemma \ref{lemma7}, given any $\veps > 0$, for $x>x_6(\veps,\eta)$, $t<t_4(\veps, \eta, x)$, the second term is less than  $\veps (t\pibar(xb_t))^{r+1} $. Recalling that $^{(r,s)}S_t =\ ^{(r+1,s)}S_t + \Delta X_t^{(r+1)}/b_t$, the third term is less than
\begin{align}\label{8b}
P\left(\Delta X_t^{(r+2)}>\delta x b_t\right) + P\left(^{(r+1,s)}S_t\ge \eta x,\  \Delta X_t^{(r+2)}\le \delta x b_t<\Delta X_t^{(r+1)} \right).
\end{align} By \eqref{tri2_0}, the second term in \eqref{8b} is less than $ \veps (t\pibar(xb_t))^{r+1}$. 
%

Hence for each $x>x_{7}$ and $t<t_4$, 
\begin{align}\label{8c}
P\left(\Delta X_t^{(r+2)}>\delta x b_t\right) 
&\le \frac{1}{(r+2)!} \left(t\pibar\left(\delta xb_t\right)\right)^{r+2} \le \frac{1}{(r+2)!} \left( C_5 \delta^{-D} t\pibar\left(xb_t\right)\right)^{r+2} \nonumber \\
      &\le  \frac{C_7}{(r+2)!} \left( C_5\delta^{-D}\right)^{r+2} x^{\rho-2} \left(t\pibar\left(xb_t\right)\right)^{r+1} \le \veps \left(t\pibar\left(xb_t\right)\right)^{r+1},
\end{align} 
by \eqref{2max0}, \eqref{lem5}, \eqref{ass9} and \eqref{bdd(c)} respectively. Here we used \eqref{ass10} to see that  $\delta^{-1}(\eta, t, x)\le C_{10} \log x/\eta$, thus $\delta^{-D(r+2)}x^{\rho-2} \to 0$ as $x \to \infty$.
 This completes the proof of \eqref{lem8}. 

\end{proof}

The next lemma gives an upper estimate for the lower tail of the trimmed process in a similar way as Lemma \ref{lemma7}.  Recall the definition of $\Delta X_t^{(s)-}$ in \eqref{1s}.
\begin{lemma}\label{lemma9}
Let $\veps >0$ and $0<\eta<1 $. If \eqref{convStb.asy} holds, for each $x > x_{7}$ and  $t \le t_4(\veps, \eta, x)$, we have 
\begin{equation}\label{lem9_0a}
		P\left( ^{(r,s )}S_t\le -\eta x ,\ \Delta X_t^{(s+1)-}\le \delta xb_t    \right)\le \veps \left(t\pibar(xb_t)\right)^{s+1}.
		\end{equation} 
Hence, we also have for each $x > x_7$ and $t \le t_4$,
\begin{equation}\label{lem9_0b}
       P(^{(r,s)}S_t \le -x)\le P\left( \Delta X_t^{(s+1)-}\ge x(1-\eta)b_t\right) + \veps \left(t\pibar(xb_t)\right)^{s+1}.
	   \end{equation}
\end{lemma}

\begin{proof}[Proof of Lemma \ref{lemma9}:]\
Let $\veps >0$ and $0<\eta<1$.
  Similar to the proof of Lemma \ref{lemma7}, the lefthand side of \eqref{lem9_0a} equals
\begin{equation}\label{9a}
P\left( ^{(r,s)}X_t -a_t\le -\eta x b_t, \Delta X_t^{(s+1)-}\le \delta xb_t    \right) 
\le P\left( ^{(r,s)}X_t -t\nu(\delta x b_t)\le -\frac 1 2\eta x b_t, \Delta X_t^{(s+1)-} \le \delta xb_t    \right).
\end{equation}
 On $\{ \Delta X_t^{(s+1)-}\le \delta xb_t \}$, recall the truncated processes in \eqref{noay0a} and \eqref{noay0b}, 
\begin{align*}
^{(r,s)}X_t = {}^{(r)}X_t + \sum_{i=1}^s \Delta X_t^{(i)-} 
 & \ge  X_t^{> - \delta x b_t} - \sum_{i=1}^r \Delta X_t^{(i)} \nonumber \\
 & \ge  \wt X_t^{\delta x b_t} - \sum_{i=1}^r \Delta X_t^{(i)}{\bf 1}\{ \Delta X_t^{(i)} < \delta x b_t\} \nonumber \\
 & \ge   \wt X_t^{\delta x b_t} -r\delta x b_t 
 \ge   \wt X_t^{\delta x b_t} -\frac 1 4 \eta x b_t.
\end{align*} 
From the argument above \eqref{7c}, we see that $ r \delta < \frac 1 4 \eta$ for $x > x_7$. This gives an upper estimate of \eqref{9a} as follows:
\begin{equation*}\label{9c}
P\left( ^{(r,s)}X_t -t\nu(\delta x b_t)\le -\frac 1 2\eta x b_t,\  \Delta X_t^{(s+1)-}\le \delta xb_t    \right) 
\le P\left(  \wt X_t^{\delta x b_t} - t\nu(\delta x b_t)\le -\frac 1 4\eta x b_t  \right).
\end{equation*}
Evaluate the expression in the same way as \eqref{7c} and \eqref{7e},  writing $ h = \delta x b_t$ to get, for any $\lambda >0$, $\epsilon > 0$,
\begin{align}\label{9d}
 P\left(X_t^{(h)}(\epsilon) \le -\frac 1 4 \eta x b_t\right)  = P\left(e^{-\lambda X_t^{(h)}(\epsilon)} \ge e^{\frac {\lambda}{4} \eta x b_t}\right)  \le  E\left(e^{-\lambda X_t^{(h)}(\epsilon)}\right)  e^{-\frac \lambda 4 \eta x b_t}.
\end{align} Similar to \eqref{7g}, noting that for each $x<0$, $e^{-x}-1+x\le e^{|x|}x^2/2$,
\begin{align}\label{9e}
E\left(e^{-\lambda X_t^{(h)}(\epsilon)}\right) &= \exp\left[t\int_{\epsilon\le |y|\le h}\left(e^{-\lambda y}-1+\lambda y {\bf 1}\{ |y|\le 1\}\right) \Pi(\rmd y)\right] \nonumber \\
                                                                  &\le \exp\left(t e^{|\lambda y|}\lambda^2/2 \int_{\epsilon\le |y| \le h} y^2 \Pi(\rmd y)\right).
\end{align} 
Then we can take $\epsilon \dto 0$ in \eqref{9e}.
By the same procedure as \eqref{7h} and \eqref{7ha}, for $x>x_7$, the last line of \eqref{9d} is no more than
\begin{align*}
&\exp\left( -\frac{\eta}{4\delta} + \frac{C_7}{2C_1} (\delta x)^{\rho-2}\right) \le \exp \left(-\frac{\eta}{4\delta}+\frac{\eta}{8\delta}\right) = \exp(-\frac{\eta}{8\delta}).
\end{align*} 
The rest of the proof of \eqref{lem9_0a} follows exactly like Lemma \ref{lemma7}. 

Next to prove \eqref{lem9_0b}, we proceed as in the proof of Lemma \ref{lemma8}. For $x > x_7$, the lefthand side of \eqref{lem9_0b} is smaller than
\begin{align}\label{9g}
P\left( \Delta X_t^{(s+1)-} \ge x(1-\eta)b_t  \right) 
                  &+ P\left(^{(r,s)}S_t\le -x,\ \Delta X_t^{(s+1)-} \le  \delta x b_t  \right) \nonumber \\
                                &+ P\left(^{(r,s)}S_t\le -x,\ \delta x b_t<\Delta X_t^{(s+1)-} < x(1-\eta)b_t   \right).
\end{align} 
Recall in \eqref{8a} we have for all $x>x_{7}$, $\delta < 1-\eta$.
By \eqref{lem9_0a}, the second term is less than $\veps \left(t\pibar(xb_t)\right)^{s+1}$.
The third term of \eqref{9g} is no more than 
\begin{align}\label{9h}
P\left(^{(r,s)}S_t\le -x,\  \Delta X_t^{(s+2)-}\le \delta x b_t<\Delta X_t^{(s+1)-} < x(1-\eta)b_t   \right) + P\left( \Delta X_t^{(s+2)-}>\delta x b_t \right).
\end{align} Recalling \eqref{tri2_0}, apply the same inequality in \eqref{lem9_0a} with $s$ replaced by $s+1$. Then the first term of \eqref{9h} is less than
\begin{equation}\label{9i}
P\left( ^{(r,s+1)}S_t\le -\eta x,\ \Delta X_t^{(s+2)-}\le \delta xb_t    \right) \le \veps P\left(t\pibar(xb_t)\right)^{s+2}.
\end{equation}  By the same argument as in \eqref{8c}, for any $x > x_7$ and $t <t_4$, the second term of \eqref{9h} is less than
\begin{align}\label{9j}
\frac{1}{(s+2)!}\left(t\pibar(\delta x b_t)\right)^{s+2}                                                                          \le \frac{C_5^{s+2}}{(s+2)!} \delta^{-(s+2)D} t\pibar(xb_t) \left(t\pibar(xb_t)\right)^{s+1} \nonumber \\                                                                          \le \frac{C_7C_5^{s+2}}{(s+2)!} \delta^{-(2+s)D}x^{\rho-2} \big( t\pibar(xb_t)\big)^{s+1} \le \veps \big(t\pibar(xb_t)\big)^{s+1}.
\end{align} Put \eqref{9h}, \eqref{9i} and \eqref{9j} together to complete the proof of \eqref{lem9_0b}.

\end{proof}


Summarizing the bounds derived in Lemma \ref{lemma6}, Lemma \ref{lemma8}, and Lemma \ref{lemma9}, we get our desired inequalities in the next lemma.
\begin{lemma}\label{lemma10}
Let $\veps > 0$, $0<\eta< 1$, and \eqref{convStb.asy} hold with $\mathcal{L}(^{(r,s )}S_t)\Rightarrow G$. Then there exists $ x_{8}(\veps, \eta, G) \ge x_7$ such that for all $x > x_{8}$,  
\begin{align}\label{lem10_0}
(1-\veps)\limsup_{t \dto 0} t\pibar^+(x(1+\eta)b_t) &\le \left\{(r+1)! (1-G(x))\right\}^{1/(r+1)} \nonumber \\
				&\le \liminf_{t\dto 0} \left( t\pibar^+(x(1-\eta)b_t)+\veps t\pibar(xb_t)\right).
\end{align} and 
\begin{align}\label{lem10_0b}
(1-\veps)\limsup_{t \dto 0} t\pibar^-(x(1+\eta)b_t) &\le \left\{(s+1)!G(-x)\right\}^{1/(s+1)}		\nonumber \\														&\le \liminf_{t\dto 0} \left( t\pibar^-(x(1-\eta)b_t)+\veps t\pibar(xb_t)\right).
\end{align} Under the same conditions, if \eqref{convStb.mod} holds with $\mathcal{L}({}^{(r)}\wt S_t) \Rightarrow\wt G$, then for $x >x_8(\veps, \eta, \wt G)$, 
\begin{align}\label{lem10_0c}
(1-\veps)\limsup_{t \dto 0} t\pibar(x (1+\eta) b_t) &\le \left\{  (r+1)!\left[ \wt G(-x)+ 1- \wt G(x))\right]  \right\}^{1/(r+1)} \nonumber \\
											& \qquad \qquad \qquad \le \liminf_{t\dto 0} \left( t\pibar(x(1-\eta)b_t)+\veps t\pibar(xb_t)\right).
\end{align}

\end{lemma}

\begin{proof}[Proof of Lemma \ref{lemma10}:] 
Fix $\veps >0$ and $0<\eta<1$.
In \eqref{2max0}, replace $x$ by $ x(1+\eta)$ and from Lemma \ref{lemma6}, for $x > x_8(\veps, \eta, G)$ such that for all $x >x_8$, $y = x\eta > y_0$, then put \eqref{2max0} and \eqref{lem6} together to get,
\begin{equation*}\label{10_a}
\frac{1-\veps}{(r+1)!}  \left(  t\pibar^+(x(1+\eta)b_t) \right)^{r+1} \le P\left(\Delta X_t^{(r+1)} \ge x(1+\eta)b_t\right)\le (1+\veps)P(^{(r,s)}S_t \ge x).
\end{equation*} Take $\limsup_{t\dto 0}$ on both sides.  By the portmanteau theorem, we get the lefthand inequality in \eqref{lem10_0}, from 
\begin{align}\label{10_a1}
\left(\frac{1-\veps}{1+\veps}\right)^{1/(r+1)} \limsup_{t\dto 0} t\pibar^+(x(1+\eta)b_t) 
&\le \left((r+1)!\limsup_{t \dto 0} P\big({}^{(r,s)}S_t \ge x \big)\right)^{1/(r+1)} \nonumber \\
&\le  \bigg((r+1)!(1-G(x))\bigg)^{1/(r+1)}.
\end{align} 

To get the righthand inequality in \eqref{lem10_0}, take $ x(1-\eta)>x_8$ in \eqref{2max0}. Then by Lemma \ref{lemma8} and \eqref{2max0}, $P(^{(r,s)}S_t \ge x)$ is less than
\begin{align*}
  P\left(\Delta X_t^{(r+1)} \ge x(1-\eta)b_t\right)+\veps \left(  t\pibar(xb_t) \right)^{r+1}
  \le \frac{\left(  t\pibar^+(x(1-\eta)b_t) \right)^{r+1}}{(r+1)!}  + \veps \left(  t\pibar(xb_t) \right)^{r+1}.
\end{align*} Take $\liminf_{t\dto 0}$ on both sides, to get
\begin{align*}
\liminf_{t\dto 0} \left(\frac{\left(t\pibar^+(x(1-\eta)b_t)\right)^{r+1}}{(r+1)!}+ \veps (t\pibar(xb_t))^{r+1}\right) &\ge \liminf_{t \dto 0} P({}^{(r,s)}S_t \ge x) \nonumber \\
&\ge  \liminf_{t \dto 0} P({}^{(r,s)}S_t > x') \quad \text{for some } x' > x , \nonumber \\
&\ge 1-G(x') \quad \text{(by portmanteau theorem)}\nonumber \\
&\to 1-G(x) \quad \text{as }x'\dto x  .
\end{align*} 
Hence we have for each $x > 0$,
\begin{equation}\label{10_b1}
\liminf_{t\dto 0} \left(t\pibar^+(x(1-\eta)b_t)+ \veps t\pibar(xb_t)\right) \ge \big( (r+1)!(1-G(x))  \big)^{1/(r+1)}.
\end{equation}
Combining \eqref{10_a1} and \eqref{10_b1}, \eqref{lem10_0} follows. 
\eqref{lem10_0b} and \eqref{lem10_0c} are proved similarly. 

\end{proof}

The next result is crucial in replacing a small $z > 0$ by $x b_t$ for appropriate $t(z,x)$ and $x$ so as to make use of the inequalities from Lemmas \ref{lemma6} to \ref{lemma10}.
\begin{lemma}\label{xbt}
Assume \eqref{convStb.asy} or \eqref{convStb.mod} holds with nondecreasing $b_t$.  For each $z > 0$ and $x > 0$, define 
\be\label{t-df}
t(z,x) := \inf \{ u > 0 : b_u \ge z/x \}.
\ee
Then $0<t(z,x) < \infty$, $t(z,x) \dto 0$ as $z/x \to 0$ and for each fixed $x > 0$,
\be\label{11_bt}
 \frac 1 z b_{t(z,x)} \to  \frac 1 x \qquad \text{as } z \dto 0.
\ee  
\end{lemma}

\begin{proof}[Proof of Lemma \ref{xbt}:]
\eqref{convStb.asy} or \eqref{convStb.mod} implies $b_t \dto 0$ as $t \dto 0$, so $0<t(z,x) < \infty$ and clearly $t(z,x) \dto 0$ as $z/x \dto 0$.
When $b_t$ is a continuous function, we have, for each $z,x >0$, $x b_{t(z,x)} = z$. Suppose $b_t$ is not continuous. We can find $n \ge 1$ such that, for all $z,x$ such that $z/x$ is small enough so that $t(z,x) \le 1$, we have
\[
 \frac{1}{n+1} < t(z,x) \le \frac 1 n.
\] Then since $b_t$ is assumed to be nondecreasing, we have
\be\label{11_bt2} b\left(1/( n + 1) \right) \le  b_{t(z,x)-} \le \frac{z}{x} \le b_{t(z,x)}  \le b\left(1 /n \right) .
\ee 
Fix $x > 1$. Then let $z \dto 0$ so that $z/x \to 0$, which implies $t(z,x) \dto 0$. Thus $n \to \infty$. Recall from Lemma \ref{bt} that \eqref{convStb.asy} or \eqref{convStb.mod} implies $b(1/(n+1))\sim b(1/n)$. Apply this fact to \eqref{11_bt2} to get \eqref{11_bt}. 

\end{proof}


\begin{lemma}\label{lemma11}
Assume \eqref{convStb.asy} or \eqref{convStb.mod} holds with nondecreasing $b_t$. Then there exists an $\alpha \ge 0$ such that, for all $y > 0$, 
\begin{equation}\label{lem11}
\lim_{z \to 0}\frac{\pibar(z)}{\pibar(zy)} = y^\alpha. 
\end{equation}
\end{lemma}

\begin{proof}[Proof of Lemma \ref{lemma11}:] 
First assume \eqref{convStb.asy} holds. Fix $ 0 < \veps, \eta < 1/2$ and $y >0$. Choose $x_{9}(\veps, \eta, y) \ge x_8$ such that  for all $x>x_{9}$, both $x$ and $ xy \ge x_{8}(\veps, \eta, G)(1+2\eta) $. Hence Lemma \ref{lemma10} applies to both $x$ and $xy$. 
Abbreviate
\begin{equation}\label{11_b_1}
 \Lambda(x):=  \left\{ (s+1)!G(-x)\right\}^{1/(s+1)}+ \left\{(r+1)!\left[ 1-G(x)\right]  \right\}^{1/(r+1)} .
\end{equation}
Note that for any real sequences $(a_n)$ and $(b_n)$, 
\[ \liminf_n a_n + \liminf_n b_n \le \liminf_{n}(a_n + b_n) \le   \limsup_{n}(a_n + b_n) \le \limsup_n a_n + \limsup_n b_n .
\]
Add \eqref{lem10_0} and \eqref{lem10_0b} in Lemma \ref{lemma10} to get, for each $x > x_{9}$, 
\begin{align}\label{lem10_00b}
(1-\veps)\limsup_{t \dto 0} t\pibar(x (1+\eta)b_t)
&\le  \Lambda(x) \le \liminf_{t\dto 0} \left( t\pibar(x(1-\eta)b_t)+2\veps t\pibar(xb_t)\right).
\end{align}
Take $x > (1+2\eta)x_{9}$, by \eqref{lem10_00b} we can choose $t_{5}(\veps, \eta, x, y)< t_4$ so small that whenever $t \le t_{5}$,
\be\label{11_a:u1}
t\pibar\left(\frac{x}{1+2\eta}(1+\eta)b_t \right) \le \frac{1+\veps}{ 1-\veps}\Lambda\left(\frac{x}{1+2\eta}\right),
\ee and 
\begin{align*}
 t\pibar\left(xyb_t\frac{1-\eta}{1-2\eta} \right)
 &\ge (1+2\veps)^{-1} \left[  t\pibar\left(xyb_t\frac{1-\eta}{1-2\eta}\right) + 2\veps t\pibar\left(xyb_t\frac{1-\eta}{1-2\eta}\right)  \right] \nonumber \\
                &\ge (1+2\veps)^{-1}\left[  t\pibar\left(xyb_t\frac{1-\eta}{1-2\eta}\right) + 2\veps t\pibar\left(xyb_t\frac 1 {1-2\eta}\right)  \right]
                \ge \frac{1-\veps}{1+2\veps} \Lambda\left( \frac{xy}{1-2\eta} \right).
\end{align*}
Take $z > 0$ and define $t(z,x)$ by \eqref{t-df}. Then by Lemma \ref{xbt} there is a $z_0$ sufficiently small that $t(z,x) \le t_{5}(\veps,\eta, x, y)$ and
\ben\label{ll_a:u2}
\frac{1+\eta}{1+2\eta} x b_{t(z,x)} \le z \le \frac{1-\eta}{1-2\eta} x b_{t(z,x)}\quad \text{whenever } z < z_0. 
\een
Then with $t = t(z,x)$, we have for $z < z_0$,
\ben\label{11_a:u3}
t\pibar(z)\le t\pibar\left(\frac{x}{1+2\eta}(1+\eta)b_t \right) \le \frac{1+\veps}{ 1-\veps}\Lambda\left(\frac{x}{1+2\eta}\right),
\een and for each $y>0$,
\be\label{11_a:u4}
t\pibar(zy)\ge t\pibar\left(\frac{xy}{1-2\eta}(1-\eta)b_t \right) \ge \frac{1-\veps}{ 1+2\veps}\Lambda\left(\frac{xy}{1-2\eta}\right).
\ee
Letting $z \dto 0$, we see that
\begin{align}\label{11_c}
\limsup_{z \dto 0} \frac{\pibar(z)}{\pibar(zy)} \le \left(\frac{1+2\veps}{1-\veps} \right)^2 \frac{\Lambda(x/(1+2\eta))}{\Lambda(xy/(1-2\eta))}.
\end{align}

In a similar way as \eqref{11_a:u1}--\eqref{11_a:u4}, we derive estimates in the other direction
and obtain the lower bound
\begin{equation}\label{11_e}
\liminf_{z \dto 0}\frac{\pibar(z)}{\pibar(zy)} \ge  \left( \frac{1-\veps}{1+2\veps} \right)^2 \frac{\Lambda(x/(1-2\eta))}{\Lambda(xy/(1+2\eta))}.
\end{equation}
Write $f_x(y) = \Lambda(x)/\Lambda(xy)$.
Since $f_x(y)$ is a nondecreasing function in $y$ for each $x$, by Helly's selection principle, there exists a sequence $\{x_n\} \to \infty$ such that for some monotone function $\Theta(\cdot)$, we have
\begin{equation}\label{R3}
\lim_{n \to \infty} \frac{\Lambda(x_n)}{\Lambda(x_n y)} = \Theta(y) \quad \text{at each continuity point $y>0$ of 
$\Theta(\cdot)$}.
\end{equation}
 Let $x=x_n(1+2\eta)$ in \eqref{11_c}, we have
\begin{equation}\label{11_f}
\limsup_{z \dto 0} \frac{\pibar(z)}{\pibar(zy)}\le \left(\frac{1+2\veps}{1-\veps}\right)^2 \frac{\Lambda(x_n)}{\Lambda(x_ny (1+2\eta)/(1-2\eta))}.
\end{equation} Let $n \to \infty$ on the RHS of \eqref{11_f} to get for $y(1+2\eta)/(1-2\eta)$ a continuity point of $\Theta(\cdot)$, 
\begin{equation*}\label{11_g}
\limsup_{z\dto 0} \frac{\pibar(z)}{\pibar(zy)}\le \left(\frac{1+2\veps}{1-\veps}\right)^2\Theta\left(y \frac{1+2\eta}{1-2\eta}\right).
\end{equation*}
If $y$ is a continuity point of $\Theta(\cdot)$, we can choose $\eta \to 0$ in a way that $y (1+2\eta)/(1-2\eta)$ is also a continuity point of $\Theta(\cdot)$. Then 
\[\limsup_{z \dto 0} \frac{\pibar(z)}{\pibar(zy)}\le \left( \frac{1+2\veps}{1-\veps}\right)^2 \Theta(y ). 
\] 
Next let $\veps \to 0$ to get 
\[ \limsup_{z \dto 0} \frac{\pibar(z)}{\pibar(zy)}\le \Theta\left(y\right).
\] Similarly choose $x = x_n(1-2\eta)$ in \eqref{11_e}, and let $n \to \infty$ and $\eta, \veps \dto 0$ in the same way to get 
\[ \liminf_{z\dto 0} \frac{\pibar(z)}{\pibar(zy)}\ge \Theta\left(y\right).
\] This shows that 
\[ \lim_{z \to 0} \frac{\pibar(z)}{\pibar(zy)} = \Theta(y) \quad \text{at each continuity point $y>0$ of $\Theta(\cdot)$}.
\] Next appeal to Feller (\cite{feller66}, Lemma VIII 8.1, p.268) to see that necessarily, $\Theta(y) = y^\alpha$ for some $\alpha \ge 0$. 
This completes the proof of \eqref{lem11} from \eqref{convStb.asy}.  

Next assume \eqref{convStb.mod} holds. We replace the definition in \eqref{11_b_1} by 
\[\wt \Lambda'(x) : = \left((r+1)!(\wt G(x) + \wt G(-x))\right)^{1/(r+1)}
\]
 and replace \eqref{lem10_00b} by \eqref{lem10_0c}. The rest of the proof remains the same. This completes the proof of Lemma \ref{lemma11}.

\end{proof}

\begin{lemma}\label{alpha}
In Lemma \ref{lemma11}, we can restrict $\alpha$ to $0<\alpha < 2$.
\end{lemma}

\begin{proof}[Proof of Lemma \ref{alpha}:] 
By \eqref{lem5} and \eqref{ass9}, there exists a constant $0<C_{12}<\infty$ such that for $y\ge 1, yz \le 1$,
\begin{equation}\label{alpha_1}
C_{12} y^{2-\rho} \le \frac{\pibar(z)}{\pibar(zy)}\le C_5 y^{D}.
\end{equation}
The lefthand inequality in \eqref{alpha_1} follows from \eqref{ass2} and \eqref{ass9}, which gives 
\[ \pibar(zy) \le C_2 \frac{V(zy)}{V(z)}\frac{V(z)}{z^2} y^{-2} \le \frac{(1+C_2)C_2}{C_1}y^{\rho-2} \pibar(z):= C_{12}^{-1} y^{\rho-2}\pibar(z).
\]
Hence \eqref{alpha_1} shows that we must have $0<\alpha < \infty$ in \eqref{lem11}.
For $\alpha > 2$, the L\'evy measure fails the integrability condition, i.e.
\[ \int_{\R_*} \,( 1 \wedge x^2  ) \, \Pi(\rmd x) = \infty.
\]
So $0< \alpha\le 2$. Last suppose $\alpha = 2$ and recall that $\sigma^2 = 0$. Then by 
 Feller \cite{feller66} VIII.9, Theorem 1, p.273, we have
\[ \lim_{x \dto 0}\frac{x^2\pibar(x)}{2\int_0^x y \pibar(y)\rmd y} = 0.
\] This implies $x^2\pibar(x)/U(x) \to 0$, so that $X_t$ is in the domain of attraction of the normal law, contrary to assumptions. We can conclude that $0<\alpha <2$. 

\end{proof}

So far we have shown that either \eqref{convStb.asy} or \eqref{convStb.mod} implies that the L\'evy measure has regularly varying singularity with index $-\alpha$ and $\alpha \in (0,2)$. Next we need to treat the two cases separately to get the limit of $\pibar^\pm(x)/\pibar(x)$ as $x \dto 0$. A complication comes when starting from assumption \eqref{convStb.mod} as the distribution of the ordered modulus jump $\wt {\Delta X}_t^{(j)}$ is expressed in terms of both $\pibar$ and $\pibar^\pm$ (see \eqref{or-1} below) whereas in asymmetrical trimming, the ordered jumps $\Delta X_t^{(j),\pm}$ only involve $\pibar^\pm$. 

\begin{lemma}\label{lemma12}
Under the conditions of Lemma \ref{lemma11}, \eqref{convStb.asy} implies the limits
\begin{equation}\label{lem12}
\lim_{z \dto 0} \frac{\pibar^{\pm}(z)}{\pibar(z)} \quad \text{exist}.
\end{equation}
\end{lemma}

\begin{proof}[Proof of Lemma \ref{lemma12}:]
Recall the definition of $\Lambda(\cdot)$ in \eqref{11_b_1} and note that for all $x > 1$,
\[\{(s+1)!G(-x)\}^{1/(s+1)}/\Lambda(x)\le 1.
\]
We proved in \eqref{R3} that for some sequence $\{x_n\}$ 
\begin{equation}\label{12_a}
\lim_{n \to \infty}\frac{\Lambda(x_n)}{\Lambda(x_n y)} = y^\alpha
\end{equation} for some $\alpha \in (0,2)$ and each $y > 0$.
By taking a further subsequence if necessary, still denoted by $\{x_n\}$, we have 
\[ \lim_{n\to \infty} \frac{\{(s+1)!G(-x_n)\}^{1/(s+1)}}{\Lambda(x_n)} = \theta \quad \text{for some } 0<\theta<1 .
\] 

Proceeding similarly to \eqref{11_a:u1}, using \eqref{lem10_0b},
take $x = x_n(1+ 2\eta)$, and sufficiently small $z > 0$ with $t = t(z,x)$ in \eqref{t-df}, we have
\begin{align}\label{12_b}
t\pibar^-(z)\le t\pibar^-\left( xb_t\frac{1+\eta}{1+2\eta}\right)&\le \left(\frac{1+\veps}{1-\veps}\right) \left((s+1)!G\left(-\frac{x}{1+2\eta}\right)\right)^{1/(s+1)} \nonumber \\
&= \left(\frac{1+\veps}{1-\veps}\right) \left((s+1)!G\left(-x_n\right)\right)^{1/(s+1)}.
\end{align} Similar to \eqref{11_a:u4}, take $y = 1-2\eta/1+2\eta$ and $x = x_n /(1+2\eta)$, and for sufficiently small $z > 0$ with $t = t(z,x)$, we have  
\begin{align}\label{12_c}
t\pibar\left(z\frac{1-2\eta}{1+2\eta}\right)\ge \left(\frac{1-\veps}{1+2\veps}\right)\Lambda\left(\frac{x}{1+2\eta}\right)=\left(\frac{1-\veps}{1+2\veps}\right)\Lambda(x_n).
\end{align} Putting \eqref{12_b} and \eqref{12_c} together,
\begin{align}\label{12_d}
 \frac{\pibar^-(z)}{\pibar(z (1-2\eta)/(1+2\eta))}\le \left(\frac{1+2\veps}{1-\veps}\right)^2 \frac{\left((s+1)!G\left(-x_n\right)\right)^{1/(s+1)}}{\Lambda(x_n)}.
\end{align} By \eqref{lem11}, we see that, for sufficiently small $z>0$,
\begin{equation}\label{12_e}
\frac{\pibar^-(z)}{\pibar(z)}\le \left(\frac{1+2\veps}{1-\veps}\right)^3\left(\frac{1+2\eta}{1-2\eta} \right)^{\alpha}\frac{\left((s+1)!G\left(-x_n\right)\right)^{1/(s+1)}}{\Lambda(x_n)}.
\end{equation} 
Now we can take $\limsup_{z\dto 0}$ on the lefthand side and $n \to \infty$, $\veps, \eta \dto 0$ on the right hand side, to achieve the upper estimate
\[ \limsup_{z\dto 0} \frac{\pibar^-(z)}{\pibar(z)}\le \theta.
\] Similarly, we also have
\[ \liminf_{z\dto 0} \frac{\pibar^-(z)}{\pibar(z)}\ge \theta.
\] This completes the proof of \eqref{lem12} for $\pibar^-$ and $\pibar^+$ is similar.

\end{proof}

\section{Extra argument for Modulus Trimming}

Lemma \ref{lemma12} completes the proof of Theorem \ref{convStb} from assumption \eqref{convStb.asy}. The next lemma starts from assumption \eqref{convStb.mod} and gives the last ingredient of the proof. Note that an extra argument is needed (see \eqref{lem13.6} -- \eqref{lem13.17}).

\begin{lemma}\label{lemma13}
Under the conditions of Lemma \ref{lemma11}, \eqref{convStb.mod} also implies \eqref{lem12}.
\end{lemma}

\begin{proof}[Proof of Lemma \ref{lemma13}:]
Recall that $\LLL(^{(r)}\wt S_t) \to\wt G$.
Fix $\veps > 0$ and $0<\eta < 1/3$.
By \eqref{lem8_0b}, for $x > x_9$,  $t < t_5$ we have 
\begin{align}\label{lem13.1}
P(\wt{\Delta X}_t^{(r+1)} \le -x(1-\eta)b_t) 
\ge P(^{(r)}\wt S_t \le -x) -\veps \left(t\pibar(xb_t)\right)^{r+1} \ge (1-\veps) \wt G(-x-)- \veps \left(t\pibar(xb_t)\right)^{r+1}.
\end{align}
By \eqref{11_bt}, for $x > x_9$ and sufficiently small $z > 0$ with $t = t(z,x)$, we have by \eqref{lem13.1},
\ben\label{lem13.1.5}
P(\wt{\Delta X}_t^{(r+1)} \le -z) \ge P\left( \wt{\Delta X}_t^{(r+1)} \le - xb_t \frac{1-\eta}{1-2\eta} \right)\ge (1-\veps)\wt G\left(-\frac{x}{1-3\eta}\right)-\veps \left(t\pibar\left(\frac{xb_t}{1-2\eta}\right)\right)^{r+1}.
\een
Write 
\[ \wt \Lambda (x) = (r+1)! \left( \wt G(-x) +  1 - \wt G(x) \right).
\]
By \eqref{11_bt} (also see \eqref{11_a:u1}-\eqref{11_a:u4}), 
\begin{align*}
\left( t\pibar(z)\right)^{r+1} 
&\le \left( t\pibar\left(xb_t\frac{1+\eta}{1+2\eta}\right) \right)^{r+1} 
\nonumber \\
&\le (1+\veps) \left(\frac{1+2\eta}{1-3\eta}\right)^{\alpha(r+1)}\left(t\pibar\left(xb_t\frac{1+\eta}{1-3\eta}\right) \right)^{r+1} \nonumber \\
& \le \frac{(1+\veps)^2}{(1-\veps)^{r+1}} \left(\frac{1+2\eta}{1-3\eta}\right)^{\alpha(r+1)} \wt \Lambda\left(\frac{x}{1-3\eta}\right),
\end{align*}
where the second line is due to \eqref{lem11} and the last line is due to \eqref{lem10_0c}.
Therefore we have for $x > x_{10}(\veps,\eta) \ge x_9$ and sufficiently small $z > 0$ with $t =t(z,x)$,
\begin{align}\label{lem13.3}
 \frac{P(\wt{\Delta X}_t^{(r+1)} \le -z)}{\left( t\pibar(z)\right)^{r+1} } 
& \ge \frac{(1-\veps) \wt G(-x/(1-3\eta))}{(t\pibar(z))^{r+1}} - \veps\frac{(t\pibar(xb_t/(1-2\eta)))^{r+1}}{(t\pibar(z))^{r+1}} \nonumber \\
& \ge \left(\frac{1-\veps}{1+\veps}\right)^{r+2} \left(\frac{1-3\eta}{1+2\eta}\right)^{\alpha(r+1)} \frac{\wt G(-x/(1-3\eta))}{\wt \Lambda(x/(1-3\eta))}- \veps.
\end{align} We can get an upper bound of the ratio in a similar way. By \eqref{lem6_2b}, with $x' = x/(1+2\eta)$ and $y = \eta x'/(1+2\eta)$ we have for sufficiently small $z>0$ with $t = t(z,x)$,
\be\label{lem13.4}
P(\wt{\Delta X}_t^{(r+1)} \le -z ) \le P(\wt{\Delta X}_t^{(r+1)} \le -x\frac{1+\eta}{1+2\eta}b_t ) \le (1+\veps)^2 \wt G(-x/(1+2\eta)). 
\ee Again by \eqref{lem11} and the RHS of \eqref{lem10_0c}, we have 
\begin{align}\label{lem13.5}
\left(t\pibar(z)\right)^{r+1} 
&\ge \left( t\pibar\left(x b_t \frac{1+3\eta}{1+2\eta}\right) \right)^{r+1} \nonumber \\
&\ge \frac{1-\veps}{(1+\veps)^{r+1}} \left( \frac{1-\eta}{1+3\eta}\right)^{\alpha(r+1)} \left[ t\pibar\left(x b_t \frac{1-\eta}{1+2\eta}\right)+ \veps t\pibar\left(x b_t \frac{1-\eta}{1+2\eta}\right)  \right]^{r+1} \nonumber \\
&\ge \frac{1-\veps}{(1+\veps)^{r+1}} \left( \frac{1-\eta}{1+3\eta}\right)^{\alpha(r+1)} \left[ t\pibar\left(x b_t \frac{1-\eta}{1+2\eta}\right)+ \veps t\pibar\left(x b_t \frac{1}{1+2\eta}\right)  \right]^{r+1} \nonumber \\
&\ge \frac{(1-\veps)^2}{(1+\veps)^{r+1}} \left( \frac{1-\eta}{1+3\eta}\right)^{\alpha(r+1)} \wt \Lambda \left( \frac{x}{1+2\eta}\right).
\end{align} 
Putting \eqref{lem13.4} and \eqref{lem13.5} together, we can achieve an upper bound as follows.
\be\label{lem13.6}
  \frac{P(\wt{\Delta X}_t^{(r+1)} \le -z)}{\left( t\pibar(z)\right)^{r+1} } \le \left(\frac{1+\veps}{1-\veps}\right)^{r+3} \left(\frac{1+3\eta}{1-\eta}\right)^{\alpha(r+1)} \frac{\wt G(-x/(1+2\eta))}{\wt \Lambda(x/(1+2\eta))}.
\ee
Next we would like to extract the information about $\pibar^-(z)$ from $P(\wt{\Delta X}_t^{(r+1)} \le -z)$. To achieve this, observe that 
$\Pi^\pm$ is absolutely continuous with respect to $\Pi^{|\cdot|}$, and define the Radon-Nikodym derivatives $g^\pm = \Pi^\pm /\Pi^{|\cdot|}$. By a similar calculation as in \eqref{rJ1} (see \cite{bfm14} for more details), we have
\begin{align}\label{or-1}
P(\wt{\Delta X}_t^{(r+1)} \le -z) 
&= \int_{y>z} g^-(y) P(|\wt {\Delta X}_t^{(r+1)} | \in \rmd y) \nonumber \\
&= \int_{y>z} g^-(y) te^{-t\pibar(y)}\frac{(t\pibar(y))^r}{r!} \Pi^{|\cdot|}(\rmd y) \quad (\text{by \eqref{rJ1}}) \nonumber \\
&\ge  \frac{t^{r+1}}{r!}e^{-t\pibar(z)} \int_{y>z} \pibar(y)^r \Pi^-(\rmd y).
\end{align}
The second line follows by noting that the image measure of Lebesgue measure  under mapping $\pibarinv$ is $(\rmd y)^{\pibarinv} = \Pi^{|\cdot|}$. The third line is due to the fact that $g^-(y)\Pi^{|\cdot|}(\rmd y) = \Pi^-(\rmd y)$. Recall from \eqref{ass9} and \eqref{11_bt} that we have for $x > x_{11}(\veps, \eta) \le x_{10}(\veps, \eta)$ and $t<t_5 $, 
\[ 
 t\pibar(z) \le t\pibar( xb_t/2)\le C_7( x/2)^{\rho-2} \le \veps \le -\log(1-\veps).
\]
Hence $e^{-t\pibar(z)}\ge 1-\veps$ in \eqref{or-1} and by \eqref{lem13.6}, 
\begin{align}\label{lem13.7}
\frac{\int_{y>z} \pibar(y)^r \Pi^-(\rmd y) }{\pibar(z)^{r+1}}
& \le \frac{r!}{ 1-\veps} \frac{P(\wt{\Delta X}_t^{(r+1)} \le -z)}{\left( t\pibar(z)\right)^{r+1} } \nonumber \\
 & \le r! \left(\frac{1+\veps}{1-\veps}\right)^{r+4} \left(\frac{1+3\eta}{1-\eta}\right)^{\alpha(r+1)} \frac{\wt G(-x/(1+2\eta))}{\wt \Lambda(x/(1+2\eta))} . 
\end{align}
Note that the LHS of \eqref{lem13.7} does not depend on $x$. Since $\wt G /\wt \Lambda $ is bounded, there exists a sequence $\{x_n\} \to \infty$ such that the limit 
\ben\label{lem13.8}
\theta : = \frac{r! \wt G(-x_n)}{\wt \Lambda(x_n) } \quad \text{exists and is in } [0,1/(r+1)].
\een 
Choose $x = x_n(1+2\eta)$ on the RHS of \eqref{lem13.7}, take $n \to \infty$ and then $\veps \dto 0$, $\eta \dto 0$ to get
\ben\label{lem13.9}
\limsup_{z \to 0} \frac{\int_{y>z} \pibar(y)^r \Pi^-(\rmd y) }{\pibar(z)^{r+1}} \le \theta.
\een
With a similar argument using instead \eqref{lem13.3}, we can obtain the same lower bound for the liminf. Putting the two together we have shown
\be\label{lem13.10}
\lim_{z \to 0} \frac{\int_{y>z} \pibar(y)^r \Pi^-(\rmd y) }{\pibar(z)^{r+1}} = \theta.
\ee
Define a measure $W(\rmd y)$ by its tail function
\[ W(z) =\int_{y>z} \pibar(y)^r \Pi^-(\rmd y) = \int_{y>z} W(\rmd y). 
\] Then we have $\Pi^-(\rmd y)  = \pibar(y)^{-r} W(\rmd y)$. Hence for each $z > 0$,
\[ \pibar^-(z) = \int_{y>z} \pibar(y)^{-r} W(\rmd y).
\] Note that $\pibar(z)^{-r}$ is a nondecreasing function in $z$ which has value $0$ at $0$. We can write
\[ \pibar(z)^{-r}  = \int_{0<y<z} \rmd(\pibar(y)^{-r}).
\]
Exchange the order of integration by Fubini's theorem to get
\begin{align}\label{lem13.11}
\frac{\pibar^-(z)}{\pibar(z)} 
&= \frac{1}{\pibar(z)} \int_{y > z} \left(\int_{0<x < y} \rmd (\pibar(x)^{-r})\right) W(\rmd y) \nonumber \\
&= \frac{1}{\pibar(z)} \int_{y > z} \left(\int_{0<x < z}  + \int_{z<x<y}\rmd (\pibar(x)^{-r})\right) W(\rmd y) \nonumber \\
& = \frac{1}{\pibar(z)} \left(  \int_{0<x < z } \int_{y > z}  W(\rmd y)\rmd (\pibar(x)^{-r}) + \int_{x >z} \int_{y > x} \rmd W(y)\rmd (\pibar(x)^{-r}) \right)\nonumber \\
&=  \frac{W(z)}{\pibar(z)^{r+1}} + \int_{x>z} W(x)\rmd (\pibar(x)^{-r}) \nonumber \\
& = \frac{W(z)}{\pibar(z)^{r+1}} + \frac{r}{\pibar(z)}\int_{x>z} \frac{W(x)}{\pibar(x)^{r+1}} \Pi^{|\cdot|}(\rmd x) .
\end{align}
We assume $\pibar(x) > 0$ for all $x > 0$ in \eqref{lem13.11}, otherwise truncate the integrals at the right extreme of $\pibar$. Recall that $\pibar(y)$ is assumed continuous, so $\rmd \pibar^{-r}(x) = \pibar^{-r-1}(x)\Pi^{|\cdot|}(\rmd x)$.
Note that by \eqref{lem13.10}, we have $W(z)/\pibar(z)^{r+1} \to \theta$ as $z \to 0$. Thus the first term in the last line of \eqref{lem13.11} converges to $\theta$.  To deal with the second term, we observe that for any given $\vsig > 0$, there exists $\vphi(\vsig) > 0$ small such that with $0< y < \vphi(\vsig)$, we have by \eqref{lem13.10} that 
\be\label{lem13.12}  
\theta (1-\vsig) < \frac{W(y)}{\pibar(y)^{r+1}} < \theta (1+\vsig).
\ee
Choose $z < \vphi(\vsig)$. Then we can write the integral in the last line of \eqref{lem13.11} as
\begin{align}\label{lem13.13}
&\frac{r}{\pibar(z)} \int_{y>z} \frac{W(y)}{\pibar(y)^{r+1}} \Pi^{|\cdot|} (\rmd y) \nonumber \\
&= \frac{r}{\pibar(z)} \int_{z}^{\vphi(\vsig) }\frac{W(y)}{\pibar(y)^{r+1}} \Pi^{|\cdot|} (\rmd y)+
\frac{r}{\pibar(z)} \int_{\vphi(\vsig)}^\infty \frac{W(y)}{\pibar(y)^{r+1}} \Pi^{|\cdot|} (\rmd y) =: \YYY_1(z, \vphi) + \YYY_2(z, \vphi).
\end{align}
As $\pibar(0+) = \infty$, $ \lim_{z \to 0} \YYY_2(z,\vphi) = 0$.
By \eqref{lem13.12}, we can bound $\YYY_1$ above and below as follows:
\begin{align*}
\YYY_1(z, \vphi)  < \frac{r}{\pibar(z)} \theta (1+ \vsig) \int_{z}^\vphi \Pi^{|\cdot|}(\rmd y)
= r\theta(1+ \vsig) \left(1 - \frac{\pibar(\vphi)}{\pibar(z)}\right)
\end{align*} and
\begin{align*}
\YYY_1(z, \vphi)  > \frac{r}{\pibar(z)} \theta (1 - \vsig) \int_{z}^\vphi \Pi^{|\cdot|}(\rmd y)
= r\theta(1- \vsig) \left(1 - \frac{\pibar(\vphi)}{\pibar(z)}\right).
\end{align*} Take $z \to 0$ to get 
\ben\label{lem13.16}
 r \theta (1- \vsig)<\lim_{z \to 0} \YYY_1(z, \vphi) < r \theta (1+ \vsig).
\een
Since limits exist for both $\YYY_1$ and $\YYY_2$, we can add them together to get 
\be\label{lem13.17}
 r \theta (1- \vsig)<\lim_{z\to 0}\frac{r}{\pibar(z)} \int_{z}^\infty \frac{W(y)}{\pibar(y)^{r+1}} \Pi^{|\cdot|} (\rmd y)<r \theta (1+ \vsig).
\ee Now take $\vsig \to 0$, then the last line of \eqref{lem13.11} tends to $\theta + r \theta$, hence 
\[ \lim_{z\to 0 } \frac{\pibar^-(z)}{\pibar(z)} = \theta (1+r) \le  1.
\]
This completes the proof.

\end{proof}

\begin{remark}\label{sldiff}
We note here a distinctive difference between our small time and Kesten's large time derivations (\cite{kesten93}) in Lemma \ref{lemma13}. The singularity in $\Pi$ at $0$ is required to eliminate $\YYY_2(z, \vphi)$ in \eqref{lem13.13}.
\end{remark}

So far we have proven that under the conditions in Theorem \ref{convStb} with diffuse L\'evy measure $\Pi$, either \eqref{convStb.asy} or\eqref{convStb.mod} implies $\pibar$ is regularly varying with index $\alpha \in (0,2)$ at $0$ (see Lemma \ref{lemma11}-- Lemma \ref{alpha}). Lemma \ref{lemma12} shows that \eqref{convStb.asy} implies the limit $\pibar^\pm(x)/\pibar(x)$ exists as $x \dto 0$ and Lemma \ref{lemma13} proves the existence of the limit from assumption \eqref{convStb.mod}. In the next section, we will remove the extra assumption on the L\'evy measure $\Pi$ to complete the proof in the most general setting.

\section{Remove the Continuity Assumption}\label{rmcty}
In this section, we aim to show that it is enough to prove Theorem \ref{convStb} with the assumption that the L\'evy measure of $X_t$, i.e. $\Pi$, is a diffuse measure. To see this, let us construct a L\'evy process $X_t^*$ with a continuous L\'evy measure $\Pi^*$ by the following procedure.

Let $(U_t)_{t \ge 0}$ be a subordinator with L\'evy measure $\Pi_U(\rmd x) = \rmd x {\bf 1}_{0\le x \le 1}$ and having jump process $(\Delta U_s)_{s \le t}$ independent of $(X_t)_{t\ge 0}$. Define $Y_t : =  \sum_{0<s\le t} \sgn (\Delta X_s) \Delta U_s  (\Delta X_s)^2$, $t > 0$. Then $Y_t$ is a L\'evy process.


Recall that by the L\'evy-It\^o decomposition, we can write $X_t$ as
\[ X_t = \gamma t + \sigma Z_t + X_t^{J}
\] where $X_t^J$ is the a.s. limit of a compensated jump process, i.e.
\[ X_t^J = \lim_{\veps \dto 0} \left(\sum_{0<s\le t } \Delta X_s {\bf 1}_{|\Delta X_s| \ge \veps} - t\int_{\veps < |x| \le 1} x \Pi(\rmd x)\right).
\]
Convolve the jump process $X_t^J$ with the randomised quadratic variation process, i.e. define 
\ben\label{sly3}
X_t^{J,*} = X_t^J + Y_t .
\een
Then let the new process be defined as 
\begin{equation*}\label{sly1}
X_t^{*} = \gamma t + \sigma Z_t + X_t^{J, *} = X_t + Y_t
\end{equation*}  where $\gamma$, $\sigma$, $Z_t$ are the same as occur in $X_t$. 

Denote $g_s = \sgn(\Delta X_s)$, $0<s\le t$, and the jumps of $X_t^*$ by $(\Delta X_s^*)_{s\le t}$. Note that the positive jumps of $X_t^*$ can only consist of the positive jumps of $ X_t$ as $\Delta X_s$ and $g_s \Delta U_s (\Delta X_s)^2$ are of the same sign. Similarly, the negative jumps of $X_t^*$ correspond to the negative jumps of $X_t$. The difference between $\Delta X_t^*$ and $\Delta X_t$ are only by magnitude but not by sign.
Hence for each $0<u<1$, $x > 0$, 
\begin{align*}\label{sly1:a}
E \sum_{0<s\le 1} {\bf 1}_{(\Delta X_s + g_s u(\Delta X_s)^2 > x)} 
=  E \sum_{0<s\le 1} {\bf 1}{\left(\Delta X_s> \frac{\sqrt{1+4ux} - 1}{2u}\right)} 
= \pibar^+\left( \frac{\sqrt{1+4ux} - 1}{2u}\right).
\end{align*}
Similarly, for $\Delta X_s < 0$,
\ben\label{sly1:b}
E \sum_{0<s\le 1} {\bf 1}_{(\Delta X_s + g_s u(\Delta X_s)^2 <- x)} 
= 
E \sum_{0<s\le 1} {\bf 1}{\left(\Delta X_s  < \frac{1-\sqrt{1+4ux} }{2u}\right)}
= \pibar^-\left( \frac{\sqrt{1+4ux} -1}{2u}\right) .
\een

Hence the L\'evy measure for $X_t^{*}$, denoted by $\Pi^*$, has tails, respectively, 
\begin{equation*}\label{sly4}
\pibar^{*,+}(x) = \int_0^1E \sum_{0<s\le 1} {\bf 1}_{(\Delta X_s + g_s u(\Delta X_s)^2 > x)} \rmd u 
 = \int_0^1  \pibar^+\left( \frac{\sqrt{1+4ux} - 1}{2u}\right) \rmd u
\end{equation*} and
\begin{equation*}\label{sly4:a}
\pibar^{*,-}(x) = \int_0^1E \sum_{0<s\le 1} {\bf 1}_{(\Delta X_s + g_s u(\Delta X_s)^2 <- x)} \rmd u 
 = \int_0^1  \pibar^-\left( \frac{\sqrt{1+4ux} +1}{2u}\right)\rmd u.
\end{equation*}
Since $\Pi$ has only a countable set of atoms, integration of its tail functions against Lebesgue measure means that $\pibar^{*,+}$ and $\pibar^{*,-}$ are continuous at each $x > 0$.

Recall we have explained in Section \ref{sect:pre} that without the assumption of continuity, \eqref{convStb.asy} or \eqref{convStb.mod} implies $S_t$ is tight at $0$. Also we eliminated the cases when $X_t$ is in the partial domain of attraction of a normal law, so $\sigma^2 = 0$, and also the degenerate case. We deduce that $S_t$ is in the Feller Class at $0$. 
By Maller and Mason 2010 \cite{MM2010} Theorem 2.1, for each subsequence $\{t_k \dto 0\}$, we then have the convergence of the bivariate L\'evy process
\ben\label{sly5}
\left(\frac{X_{t_{k'}}-a_{t_{k'}}}{b_{t_{k'}}} ,\,  \frac{V_{t_{k'}}}{b^2_{t_{k'}}} \right) \todr \left(\III'  , \, \JJJ'  \right),
\een through a further subsequence, denoted by $t_{k'} \dto 0$, where $'$ indicates that the limit depends on the choice of the subsequence. Here $\JJJ'$ is the quadratic variation process corresponding to the L\'evy process constructed from $\III'$.
Hence both $\III'$ and $\JJJ'$ are a.s. finite random variables. 
This implies that the quadratic variation process is of order $b_t^2$ as 
$t \dto 0$ i.e. 
\ben\label{sly4.5}
 V_t :=  \sum_{0< s \le t} (\Delta X_s)^2 = O_p(b_t^2).
\een

Now under the conditions of Theorem \ref{convStb}, let \eqref{convStb.asy} hold. Then observe that
\ben\label{sly6}
{}^{(r,s)}S_t^* := \frac{{}^{(r,s)}X_t^{*} -a_t}{b_t} 
= {}^{(r,s)}S_t +\frac{{}^{(r,s)}X^*_t - {}^{(r,s)}X_t}{ b_t}. 
\een
For each $r \in \N$, let $\Delta X_t^{*(r)\pm} $ be the $r^{th}$ largest positive and negative jumps in $X_t^*$ up till time $t$. Also denote by $\wt{\Delta X}_t^{*,(r)}$ the $r^{th}$ largest modulus jump in $X_t^*$ up till time $t$. 
By construction in \eqref{sly3}, for each $0 < s\le t$, we have $\sgn(\Delta X_s^*) = \sgn(\Delta X_s)$ and $|\Delta X_s^*| = |\Delta X_s| + \Delta U_s  |\Delta X_s|^2 $. Hence $\Delta X_t^{*(r),\pm} \ge \Delta X_t^{(r),\pm}$ a.s. and $|\wt{\Delta X}_t^{*,(r)}|\ge |\wt{\Delta X}_t^{(r)}|$ a.s. with $\wt{\Delta X}_t^{*,(r)}$ having the same sign as $\wt{\Delta X}_t^{(r)}$.
Note that the jumps $(\Delta U_s)_{s\le t}$ of $U$ lie in $[0,1]$. So 
\begin{equation*}\label{sly6:u1}
0\le |\wt{\Delta X}_t^{*,(r)}-\wt{\Delta X}_t^{(r)}| 
 \le (\wt{\Delta X}_t^{(r)})^2 \le V_t = O_p(b_t^2).
\end{equation*}
Similarly, 
\ben\label{sly6:u2}
0\le \Delta X_t^{*(r),\pm} -\Delta X_t^{(r),\pm}  = O_p(b_t^2).
\een
Also 
\be\label{sly6:u3}
0\le |X_t^* - X_t| = |Y_t| \le V_t = O_p(b_t^2).
\ee
From these we conclude that as $t \dto 0$, 
\begin{align*}\label{sly.6a} 
\frac{|{}^{(r,s)}X_t^* -{}^{(r,s)} X_t|}{b_t}
\le \frac{V_t}{b_t} = O_p(b_t) \topr 0 \quad \text{and}\quad
\frac{|{}^{(r)}\wt X_t^* -{}^{(r)} \wt X_t|}{ b_t}  = O_p(b_t) \topr 0.
\end{align*}
Therefore \eqref{convStb.asy} or \eqref{convStb.mod} implies that ${}^{(r,s)}S_t^*$ or ${}^{(r)}\wt S^*_t$ also converges as $t \dto 0$. Since $X_t^*$ has continuous L\'evy measure $\Pi^*$, by our proof in Section \ref{sect:prf}, $S_t^*$ converges as $t \dto 0$. 
Therefore by \eqref{sly6:u3}, $S_t$ also converges as $t \dto 0$.
In this way, we have established the result in general. This finally completes the proof of Theorem \ref{convStb}. 
\halmos

\bigskip \noindent{\bf Acknowledgement.}
The author is very grateful to Prof. Ross Maller and Dr. Boris Buchmann for critically reading the manuscript and giving helpful suggestions. The author also thanks Prof. Alan Sly for suggesting to explore the idea in the last section.


\bibliography{library_levy} 
\bibliographystyle{plain}

\end{document}